\documentclass[hidelinks,onefignum,onetabnum]{siamart220329}

\makeatletter
\let\siam@class@refstepcounter\refstepcounter
\def\refstepcounter{\@ifnextchar[{\cref@patched@refstepcounter}{\siam@class@refstepcounter}}
\def\cref@patched@refstepcounter[#1]#2{\siam@class@refstepcounter{#2}}
\providecommand\cref@override@label@type[2]{}
\makeatother

\newcommand{\poploss}{F_{\D}}
\newcommand{\xsimplex}{\Delta_x}
\newcommand{\ysimplex}{\Delta_y}
\newcommand{\lips}{L_0}
\newcommand{\smoo}{L_1}

\newcommand{\X}{\mathcal{X}}
\newcommand{\Y}{\mathcal{Y}}
\newcommand{\T}{\mathcal{T}}

\newcommand{\D}{\mathcal{D}}
\newcommand{\Z}{\mathcal{Z}}
\newcommand{\A}{\mathcal{A}}
\newcommand{\F}{\mathcal{F}}
\newcommand{\RR}{\mathbb{R}}

\newcommand{\gap}{\operatorname{Gap}}
\newcommand{\lse}{\operatorname{LSE}}
\newcommand{\PP}{\mathbb{P}}

\usepackage{multirow}
\usepackage{bbm}
\usepackage{amssymb}
\usepackage{lipsum}
\usepackage{amsfonts}
\usepackage{algorithmic}
\ifpdf
  \DeclareGraphicsExtensions{.eps,.pdf,.png,.jpg}
\else
  \DeclareGraphicsExtensions{.eps}
\fi

\newsiamremark{remark}{Remark}
\newsiamremark{hypothesis}{Hypothesis}
\crefname{hypothesis}{Hypothesis}{Hypotheses}
\newsiamthm{claim}{Claim}

\allowdisplaybreaks

\headers{Mirror Descent for DP Stochastic Saddle-Points}{T. Gonz\'alez, C. Guzm\'an, and C. Paquette}

\title{Mirror Descent Algorithms with Nearly Dimension-Independent Rates for Differentially- Private Stochastic Saddle-Point Problems}

\author{Tom\'as Gonz\'alez\thanks{Machine Learning Department, Carnegie Mellon University 
  (\email{tcgonzal@cs.cmu.edu}).}
\and Crist\'obal Guzm\'an\thanks{Institute for Mathematical and Computational Engineering, Faculty of Mathematics and School of Engineering,
 Pontificia Universidad Cat\'olica de Chile and
 Google Research 
  (\email{crguzmanp@mat.uc.cl}
  ).}
\and Courtney Paquette\thanks{Department of Mathematics and Statistics, McGill University and Google DeepMind\\(\email{courtney.paquette@mcgill.ca}).}}

\begin{document}

\maketitle

\begin{abstract}
    We study the problem of differentially-private (DP) stochastic (convex-concave) saddle-points in the $\ell_1$ setting. We propose $(\varepsilon, \delta)$-DP algorithms based on stochastic mirror descent that attain nearly dimension-independent convergence rates for the expected duality gap, a type of guarantee that was known before only for bilinear objectives. For convex-concave and first-order-smooth stochastic objectives, our algorithms attain a rate of $\sqrt{\log(d)/n} + (\log(d)^{3/2}/[n\varepsilon])^{1/3}$, where $d$ is the dimension of the problem and $n$ the dataset size. Under an additional second-order-smoothness assumption, we show that the duality gap is bounded by $\sqrt{\log(d)/n} + \log(d)/\sqrt{n\varepsilon}$ with high probability, by using bias-reduced gradient estimators. This rate provides evidence of the near-optimality of our approach, since a lower bound of $\sqrt{\log(d)/n} + \log(d)^{3/4}/\sqrt{n\varepsilon}$ exists. Finally, we show that combining our methods with acceleration techniques from online learning leads to the first algorithm for DP Stochastic Convex Optimization in the $\ell_1$ setting that is not based on Frank-Wolfe methods. For convex and first-order-smooth stochastic objectives, our algorithms attain an excess risk of $\sqrt{\log(d)/n} + \log(d)^{7/10}/[n\varepsilon]^{2/5}$, and when additionally assuming second-order-smoothness, we improve the rate to $\sqrt{\log(d)/n} + \log(d)/\sqrt{n\varepsilon}$. Instrumental to all of these results are various extensions of the classical Maurey Sparsification Lemma \cite{Pisier:1980}, which may be of independent interest.
\end{abstract}

\begin{keywords}
Differential Privacy, Stochastic Saddle Point Problem, Mirror Descent, Sparse Approximation
\end{keywords}

\begin{MSCcodes}
90C47, 90C15, 62D99, 68P27
\end{MSCcodes}

\section{Introduction}\label{sec:intro}

In this work, we study the \textit{stochastic saddle-point (SSP) problem}, that is, we aim to find a pair $(x^*,y^*)$ 
that solves 
\begin{equation} \label{eq:SSP}
    \min_{x \in \X} \max_{y \in \Y} \{ F_\D(x,y) := \mathbb{E}_{z \sim \D}[f(x,y;z)]\},
\end{equation}
where the distribution $\D$ over the sample space $\Z$ is unknown, and only iid data from the target distribution, $S \sim \D^n$, are available. We assume that, for each $z \in \Z$, the function $(x,y) \to f(x,y;z)$ is convex-concave, and the function is $\lips$-Lipschitz and $\smoo$-smooth w.r.t~$\|\cdot\|_1$ (see Section~\ref{sec: prelims} for details). 
We also assume that the feasible sets $\X$ and $\Y$ are the standard simplices in $\RR^{d_x}$ and $\RR^{d_y}$, denoted by $\Delta_x$ and $\Delta_y$ from now on\footnote{Note that we can reduce SSP problems over polytope feasible sets to this setting. In this case, the dimensions $d_x$ and $d_y$ will correspond to the number of vertices of $\X$ and $\Y$, respectively.}. We will refer to SSP problems satisfying these assumptions as the 
`$\ell_1$ setting' \cite{JNLS:2009,NesterovNemirovski:2013}. See Section \ref{sec:examples} for important problems where these assumptions hold. 

In modern machine learning applications, the observed dataset $S$ usually contains sensitive user information. In this context, Differential Privacy (DP) has become the de facto standard for privacy-preserving algorithms. A randomized algorithm $\mathcal{A}(S)$ is $(\varepsilon,\delta)$\textit{-differentially private} if for any pair of datasets $S$ and $S'$ differing in at most one data point and any event $\mathcal{E}$ in the output space, 
\[\mathbb{P}[\mathcal{A}(S)\in\mathcal{E}] \leq e^{\varepsilon}\mathbb{P}[\mathcal{A}(S')\in \mathcal{E}] + \delta.\] 

We study the Differentially Private Stochastic Saddle Point (DP-SSP) problem in the $\ell_1$ setting. More concretely, we provide differentially private algorithms that receive $S$ as input and return a pair $(x_{\A(S)},y_{\A(S)})$ that approximates an SSP. Given a pair $(x,y)$, we measure its efficiency by the duality gap, defined as 
\begin{equation} \label{eq:duality_gap}
    \operatorname{Gap}(x,y) = \max_{v \in \mathcal{X}, w \in \mathcal{Y}}(\poploss(x,w) - \poploss(v,y)).
\end{equation}

There is a major discrepancy in the literature between the achievable rates of convergence for the duality gap between some structured problems and more general convex-concave settings
. While some bilinear problems have convergences rates 
which scale only poly-logarithmically with the ambient dimensions 

(e.g.,
~\cite{Hsu:2013}), all known rates for general convex-concave objectives scale polynomially with the dimension (e.g.,~\cite{pmlr-v195-bassily23a}). This can be attributed to the fact that all existing results for convex-concave settings work under Euclidean geometry, which assumes general convex and compact feasible sets and objectives that are Lipschitz and smooth with respect to $\|\cdot\|_2$, as opposed to the $\ell_1$ setting considered in this work. Therefore, known polynomial dimension lower bounds for DP Stochastic Convex Optimization (SCO) in the Euclidean setting \cite{Bassily:2014} apply.
Recently, \cite{asi2021private,bassily2021differentially} showed that it is possible to obtain excess risk bounds for DP-SCO under $\ell_1$ setting which are nearly independent of the dimension. Inspired by these recent developments, in this work we provide DP algorithms with nearly dimension-independent rates for convex-concave SSP problems in the $\ell_1$ setting. Our results extend considerably the class of SSP problems where it is possible to bypass polynomial in the dimension lower bounds, which so far only contain problems with bilinear objectives and simplex feasible sets. 

\subsection{Summary of Contributions} We prove the first nearly-dimension independent rates for convex-concave DP-SSP in the $\ell_1$ setting. These rates are optimal (up to logarithmic factors) for second-order-smooth objectives, and have the form $r_{NDP} + r_{DP}$ where $r_{NDP}$ is the optimal non-private SSP rate  and $r_{DP}$ is an error rate due to differential privacy. Furthemore, the algorithms achieving these rates have an $O(n)$ gradient complexity. We also use our techniques to show nearly-dimension independent convergence rates for DP-SCO under the same assumptions considered in \cite{asi2021private,bassily2021differentially}. Previous work used private versions of the variance-reduced stochastic Frank-Wolfe algorithm. By contrast, we are the first to achieve this using variants of mirror-descent algorithms. 
Table \ref{tab:results_summary} summarizes the obtained rates. 

\begin{itemize}
    
    \item 
    We propose two algorithms for DP-SSP, based on (entropic) Stochastic Mirror Descent (SMD). As opposed to previous results for bilinear settings where privatization by sampling from vertices leads to unbiased gradient estimators \cite{Hsu:2013}, the main challenge we encounter is controlling the bias in the more general (stochastic) convex-concave setting. The first algorithm (Section \ref{subsec:DPSSPs_nobiasred}) controls bias by simply using a large number of samples in every iteration. 
    The second algorithm
    (Section \ref{subsec:DPSSPs_biasred}) builds upon the first one by incorporating a stochastic simulation   technique from    \cite{Blanchet:2015} to reduce the bias of gradients evaluated at the private estimators of the iterates.
    
    \item 
    In our algorithms, the iterates are privately estimated
    by sampling vertices proportionally to 
    their weights, a technique known as Maurey's Sparsification Lemma \cite{Pisier:1980}\footnote{This method is also known as the Maurey Empirical Method.}. 
    We provide various extensions of this classsical result, most notably for the duality gap and gradient sampling from vertices, 
    that are new and we believe are of independent interest. These extensions are challenging, primarily due to the presence of bias in the underlying estimation task. 
    See Section \ref{sec:Maurey} and Subsection \ref{subsec:convergenceofSMD} for more details. 

    \item 
    We give the first mirror descent style algorithm for DP-SCO in the polyhedral setting with nearly dimension-independent excess risk guarantees by combining the ideas we used for DP-SSP with the Anytime Online-to-Batch Conversion from \cite{anytime-online-to-batch}. See Section \ref{sec:DPSCO}. 
    While these rates are not optimal, we expect they could inspire future work into  more practical algorithms.
\end{itemize}

\begin{table}
\centering
\begin{tabular}{|p{1.1cm}|p{2cm}|p{3.4cm}|p{2.7cm} p{1.6cm}|}
\hline
\centering{Setting} & \centering{Objective} & \centering{Previous best-known rate} & \multicolumn{2}{|c|}{\centering{Our rate}}\\
\hline
\hline
\multirow{3}{*}{SSP} 
        & Bilinear & $\frac{\sqrt{\ell}}{\sqrt{n}} + \left(\frac{\ell^{3/2}}{n\varepsilon}\right)^{1/2}\text{\tiny{(*)}}$ \hfill{\scriptsize \cite{Hardt:2010}}  
        &$\color{black} \frac{\sqrt{\ell}}{\sqrt{n}} + \left(\frac{\ell^{3/2}}{n\varepsilon}\right)^{1/2}$\tiny{(*)} & {\footnotesize (Thm. \ref{thm:convergence_DPSSP_no_bias_reduction})}\\
        \cline{2-5}
        & Quadratic & \hspace{.5cm} None & $\color{black} \frac{\sqrt{\ell}}{\sqrt{n}} + \left(\frac{\ell^{3/2}}{n\varepsilon}\right)^{1/2}$\tiny{(*)} & {\footnotesize (Thm. \ref{thm:convergence_DPSSP_no_bias_reduction})}\\
        \cline{2-5}
        & $L_2$-smooth & \hspace{.5cm} None &$\color{black} \frac{\ell}{\sqrt{n}} + \left(\frac{\ell^{2}}{n\varepsilon}\right)^{1/2}$\tiny{(**)} & \footnotesize{(Thm. \ref{thm:convergence_bias_reduction})}\\
        \cline{2-5}
        &$L_1$-smooth & \hspace{.5cm} None & {$\frac{\sqrt{\ell}}{\sqrt{n}} + \left(\frac{\ell^{3/2}}{n\varepsilon}\right)^{1/3}$} &  {\footnotesize (Thm. \ref{thm:convergence_DPSSP_no_bias_reduction})}\\
\hline
\multirow{3}{*}{{SCO}} & 
        $L_2$-smooth& \hspace{.5cm} None & $\frac{\sqrt{\ell_x}}{\sqrt{n}} + \frac{\ell_x}{\sqrt{n\varepsilon}}$ &  \footnotesize{(Thm. \ref{thm:accuracy_anytime_method_L2smoothness})}\\
        \cline{2-5}
        & $L_1$-smooth  & $\color{black}\frac{\sqrt{\ell_x}}{\sqrt{n}} + \frac{\ell_x^{2/3}}{(n\varepsilon)^{2/3}}$\tiny{(*)}\hfill{\scriptsize\cite{asi2021private}}&  $\frac{\sqrt{\ell_x}}{\sqrt{n}} + \frac{\ell_x^{7/10}}{(n\varepsilon)^{2/5}}$ & \footnotesize{(Thm. \ref{thm:accuracy_anytime_method_L2smoothness})} \\
\hline
\end{tabular}
\caption{Rates derived from our results. We omit polylog($n,1/\delta$) factors, function-class parameters and constants. We refer to first-order-smooth and second-order-smooth functions as $L_1$-smooth and $L_2$-smooth, respectively. We denote $\ell = \log(d_x) + \log(d_y)$ and $\ell_x = \log(d_x)$. (*): these rates are optimal up to polylog factors. (**): this rate holds with 
high probability while the rest hold in expectation. } 
\label{tab:results_summary}
\end{table}

\subsection{Related Work}\label{subsec:intro/related_work}

\paragraph{DP stochastic saddle-points} Stochastic (non-private) saddle-point problems have connections to optimization, economics, machine learning, and many other fields. One such algorithm for solving these problems is the stochastic entropic mirror-descent algorithm \cite{JNLS:2009}, which is the basis for much of this work. Recently, a line of research has studied the differentially private (DP) version of stochastic saddle-point problems \cite{boob2024optimal,DPSMO, DPSGDA,Kang:2022,pmlr-v195-bassily23a}. These works focus exclusively on Euclidean setups, which necessarily leads to algorithms that have a polynomial dependence on the dimension. Moreover, due to the Euclidean setup, the techniques used in these works do not necessarily apply to the setting considered here. In addition, with only limited exceptions \cite{boob2024optimal,pmlr-v195-bassily23a}, most previous work provide guarantees on the {\it weak} duality gap. The (strong) duality gap \eqref{eq:duality_gap}, on the other hand, is more meaningful (see  \cite{pmlr-v195-bassily23a} for justification), but also more difficult to bound, most notably in the private setting. This is the convergence criterion that we use in this work.

\paragraph{Optimization in the $\ell_1$ setting}  Closest to this work is DP stochastic convex optimization (SCO) in the $\ell_1$ setting \cite{asi2021private,bassily2021non,Han:2022}, as well as earlier work on private empirical risk minimization (ERM)
\cite{Jain:2014,Talwar:2014,Talwar:2015}. These works attain nearly dimension-independent excess risk rates with private variants of the (stochastic) Frank-Wolfe method. 
Even though this method leads to optimal rates for DP-SCO, it is not known to yield any nontrivial rates for saddle-point problems. 
Even without DP constraints, it is not known how to solve SSP problems with Frank-Wolfe methods unless strong additional assumptions are made: namely, either simultaneously assuming smoothness and strong convexity/strong concavity\footnote{These assumptions are indeed strong: If $f:
\RR^d\mapsto \RR$ is an $L_1$-smooth and $\mu$-strongly convex function w.r.t.~$\|\cdot\|_1$, then $L_1/\mu=\Omega(d)$. This follows by conjugate duality and the argument in \cite[Example 4.1]{dAspremont:2018}}, 
 or strong convexity of the feasible sets \cite{Gidel:2017}.
Clearly, 
these assumptions are not suitable for this work.

\paragraph{Maurey Sparsification} This is a classical technique in high-dimensional convex geometry \cite{Pisier:1980}. The idea is based on sampling vertices with the barycentric coordinates weights and taking the empirical average, as nicely illustrated in the introduction of the book \cite{Vershynin:2018}.
This approximation is established for norms with  concentration estimates for independent random sums, e.g.,~quantified by their Rademacher type \cite{Pisier:2016}. This technique has been used systematically as a compression-type argument to obtain faster rates in statistical estimation and learning (e.g., ~\cite{Barron:1993,Zhang:2002}). More recently, \cite{Acharya:2019} proved a result similar to Lemma \ref{lem:maurey}, a functional approximation for randomized sparsification, in the context of model compression for distributed learning of generalized linear models. 
They also proposed a sparsified mirror descent algorithm that works in the $\ell_1/\ell_q$ setting where $q<\infty$, resorting on smoothness properties of the Bregman divergence. Furthemore, they only sparsify the outputs of each machine. We do not follow this approach. Instead, we sparsify every iterate of mirror descent. 
Furthermore, we show that entropic SMD (using a nonsmooth Bregman divergence) with iterate sparsification works for convex and smooth losses, which to our knowledge is an entirely new result, and does not follow from direct ex-post model compression.

\paragraph{Debiasing} This is a technique proposed in the stochastic simulation literature \cite{Blanchet:2015}, to leverage biased estimators whose accuracy improves with larger sample size. The main difference between the works on simulation and our approach is the finite sample size, which limits the amount of bias-reduction. The work 
\cite{Asi:2021bias} studied bias-reduction in regularization for SCO.
In concurrent work \cite{Ghazi:2024}, a gradient debiasing procedure is developed for Euclidean DP-SCO. Although our results do not depend on their techniques, we borrow some inspiration from them. Note futher that most of our results do not use bias-reduction.

\paragraph{Concurrent and follow up work} In independent and concurrent work, the paper \cite{Zhou:2024} investigates empirical and population worst-group risk minimization under DP, with a focus on group fairness and robust learning. While the convex-linear loss functions of their work correspond to a particular case of ours (convex-concave), their use of Euclidean geometry makes their results incomparable to ours. Furthermore, their sequential regularization approach is independent from ours. The follow up work \cite{BGM24noneuclideanDP-SSP} considers DP-SSPs in $\ell_p/\ell_q$ setup with $p, q \in [1,2]$. In contrast to our work, their results focus exclusively in the nonsmooth setting, leading to unavoidable polynomial-in-the-dimension convergence rates for the strong duality gap.

\subsection{Preliminaries}\label{sec: prelims}

Given $x \in \RR^d$, we denote the $1$-norm and $\infty$-norm by $\|x\|_1 = \sum_{j\in[d]} |x_j|$ and $\|x\|_{\infty} = \max_{j\in[d]} |x_j|$, respectively. A differentiable function $F: \RR^{d} \to \RR$ is $L_0$-Lipschitz w.r.t.~$\|\cdot\|_1$ over a set $\mathcal{X}$ if $\|\nabla F(x)\|_{\infty} \leq L_0$ for all  $x \in \X$. It is $L_1$-smooth w.r.t.~$\|\cdot\|_1$ over $\mathcal{X}$ if $\|\nabla F(x^1) - \nabla F(x^2)\|_{\infty} \leq L_1\|x^1-x^2\|_1$ for all  $x^1,x^2 \in \X$. We say a function is first-order-smooth if it is $L_0$-Lipschitz and $L_1$-smooth w.r.t.~$\|\cdot\|_1$. We say a function is second-order-smooth if it is first-order-smooth and 
its partial derivatives are $L_2$-smooth w.r.t~$\|\cdot\|_1$.\footnote{Note that this assumption is milder than the typical second-order-smoothness assumption in optimization, and it is satisfied by the widely used quadratic and logistic losses.} For $a,b:\mathbb{R}^p\mapsto\mathbb{R}$, we say $a \lesssim b$ if there exists an absolute constant $C>0$ such that for all $x\in\mathbb{R}^p$, $a(x) \leq Cb(x)$. When $a \lesssim b$ and $b \lesssim a$, we write $a \eqsim b$. We denote by $e_{i}$ the $i$-th canonical vector (the dimension being clear from context). The standard simplex in $\RR^d$ is the set $\Delta_d = \{x \in \RR^d: \|x\|_1 = 1, x_j \ge 0 \text{ for all } j\in[d]\}$. A vector $x\in\Delta_d$ 
induces a probability distribution over the canonical vectors, which we denote by $P_{x}$. Hence, $\hat{x} \sim P_{x}$ means that $\hat{x}$ takes the value $e_i$ with probability $x_i$. Note that if $\hat{x}\sim P_x$, then $\mathbb{E}[\hat{x}] = x$. Given $f:{\cal X}\times{\cal Y}\mapsto\RR$, we let $\Phi(x,y)=\big(\nabla_x f(x,y), - \nabla_y f(x,y) \big)$ be saddle operator (a.k.a.~the {\em monotone operator associated to the saddle-point problem}) \cite{Rockafellar:1970}. We denote $\ell = \log(d_x) + \log(d_y)$, where $d_x, d_y$ are the ambient dimensions of the optimization variables in the SSP problem.

\subsection{Motivating Examples}\label{sec:examples} 
As motivating examples for the DP-SSP problem \eqref{eq:SSP}, we present two problems of interest in privacy-preserving machine learning and data analysis that fit into our problem setup.

\paragraph{DP Synthetic Data Generation} 
Consider a finite sample space $\mathcal{Z} = \{z_1,...,z_{|\mathcal{Z}|}\}$ and an unknown distribution $\mathcal{D}$ supported on $\mathcal{Z}$. Suppose we have a dataset $S = \{z_{i_1},...,z_{i_n}\} \overset{iid}{\sim} \mathcal{D}$ and a set of queries $\mathcal{Q}= \{q: \mathcal{Z} \to [-1,1]\}$. Defining $q(S) = \frac{1}{|S|}\sum_{z \in S} q(z)$ and $q(\mathcal{D}) = \mathbb{E}_{z \sim \mathcal{D}}[q(z)]$, our goal is to construct a {
\em synthetic dataset} $\tilde{S}$ such that $\Tilde{S}$ is private with respect to $S$ and $\max_{q\in {\cal Q}}|q(\tilde{S}) - q(\mathcal{D})|$ is as small (in expectation) as possible. Using a linear relaxation for the discrete maximum over ${\cal Q}$ and viewing each query $q$ as vector $(q(z))_{z\in \Z}$, we obtain a bilinear SSP formulation as follows: 
\begin{equation}\label{eqn:synthetic_data_formulation}
    \min_{x \in \Delta_{|\mathcal{Z}|}} \max _{y \in \Delta_{|\mathcal{Q}|}} \mathbb{E}_{z \sim \mathcal{D}} \Big[\sum_{j \in [|\mathcal{Q}|]} y_j (q_j(z) - \langle q_j, x\rangle)\Big].
\end{equation}
While our Algorithm \ref{Alg:DPSSP_no_bias_reduction}, proposed in Section \ref{subsec:DPSSPs_nobiasred}, can solve this problem, applying it directly would result in a slightly suboptimal dependence on the histogram size, $|{\cal Z}|$. 
However, a variant of it that leverages the specific structure of this problem attains the optimal rates  \cite{lower_bound_DP_query_answering}. More concretely, the following theorem holds.

\begin{theorem}\label{thm:synthetic_data}
   Let $0 < \delta < 1$, $0 < \varepsilon < 8\log(1/\delta)$. 
    If $|\mathcal{Q}| \leq |\mathcal{Z}|^C$, then a variant of Algorithm \ref{Alg:DPSSP_no_bias_reduction} constructs an $(\varepsilon, \delta)$-DP synthetic dataset $\Tilde{S}$ of $n$ samples, such that
    \[\mathbb{E}\Big[\max_{q\in \mathcal{Q}} |\mathbb{E}_{z\sim \mathcal{D}}[q(z)] - q(\Tilde{S})|\Big] \lesssim
    \sqrt{\frac{\log(|{\cal Q}|)}{n}} + (1+C)\frac{(\log(1/\delta)\log(|\Z|))^{1/4}\sqrt{\log(|{\cal Q}|)}}{\sqrt{n\varepsilon}}.\]
    Furthermore, this rate is optimal up to constant factors.
\end{theorem}

\paragraph{Minimizing the maximal loss under DP} Consider a collection of convex losses $f_i : \X\times \Z \to [-B, B]$, $i\in [d_y]$. Suppose $\X$ is a polytope. We are interested in minimizing  
$\max_{i \in [d_y]} \{F_i(x) := \mathbb{E}_{z\sim \D}[f_i(x;z)]\}$. This can be formulated as follows: 
\[\min_{x \in \mathcal{X}} \{\max_{i \in [d_y]} F_i(x)\} = \min_{x \in \mathcal{X}} \max_{y \in \ysimplex} \mathbb{E}_{z\sim \D}\left[\sum_{i \in [d_y]} y_i f_i(x;z)\right].\]
The maximal loss formulation encompasses applications to fairness \cite{Jagielski:2019,Fair_FL,Lowy:2023} and robustness \cite{Mohajerin:2018} under privacy constraints. As a remark, if the functions $(f_i)_{i\in [d_y]}$ are $L_0$-Lipschitz and $L_1$-smooth on $x$ w.r.t~$\|\cdot\|_1$, then the problem objective $f(x,y;z) = \sum_{i \in [d_y]} y_i f_i(x;z)$ is convex-concave, $\max\{L_0, B\}$-Lipschitz,  $\max\{L_0, L_1\}$-smooth on $(x,y)$. If the functions $(f_i)_{i\in [d_y]}$ are also $L_2$-second-order-smooth on $x$, then $f(x,y)$ is $\max\{L_1, L_2\}$-second-order-smooth on $(x,y)$. Examples of convex losses that are Lipschitz, smooth and second-order-smooth over compact sets
are the quadratic and the logistic loss. 

\section{Maurey Sparsification with Functional Value Approximation}\label{sec:Maurey}

The classical Maurey Sparsification Lemma \cite{Pisier:1980} is an application of the  probabilistic method that leads to an approximate and (nearly) dimension-independent version of  Carath\'eodory's Theorem \cite{Carathodory:1911}. While this method is well-known to apply with smooth norms, it appears that for norms such as $\|\cdot\|_1$ such approximations are out of reach.\footnote{In fact, $\|\cdot\|_1$ is as far as being smooth as possible, in the sense of having trivial Rademacher type \cite{Pisier:2016}. Furthermore, polynomial in $d$ mean estimation lower bounds with respect to $\|\cdot\|_1$ are  known \cite{Waggoner:2015}.} We circumvent this limitation by directly working with function value approximations. These results will be crucial for the analysis of our algorithms in the forthcoming sections.

We start with a simple adaptation of Maurey's method for function approximation in the general smooth case. A similar proof 
appears in \cite{Acharya:2019}. Throughout this section we denote $\bar x^T = \sum_{t = 1}^{T} \frac{x^t}{T} $ and $\bar a^T = \sum_{t = 1}^{T} \frac{a^t}{T} $ unless explicitly said otherwise.
\begin{lemma}[Function value approximation in expectation]\label{lem:maurey}
    Let $(\mathbf{E},\|\cdot\|)$ be a normed space, and let $F$ be a $L_1$-smooth function w.r.t.~$\|\cdot\|$. Let $x^1,...,x^T$ be a collection of deterministic vectors, $(\lambda_1,\ldots,\lambda_T) \in \Delta_T$, and $a^1,...,a^T$ be random vectors such that $\mathbb{E}[a^t | a^{t-1},..,a^1] = x^t$. Let $\bar{x}^T =  \sum_{t = 1}^{T} \lambda_t x^t$ and $\bar{a}^T = \sum_{t = 1}^{T} \lambda_t a^t$. Then $\big|\mathbb{E}\left[F\left( \bar a^T\right) - F\left( \bar x^T\right)\right]\big|  \leq  \frac{L_1}{2}\sum_{t=1}^T\lambda_t^2\mathbb{E}\left[\left\|a^t - x^t\right\|^2\right]$.
\end{lemma}

\begin{proof}
    Let $\alpha_t = \Big\langle \nabla F \Big( \sum_{k = t}^T \lambda_k x^k +  \sum_{k = 1}^{t-1} \lambda_k a^k \Big), \lambda_t (a^{t} - x^{t}) \Big\rangle$. Noting that $F(\bar a^T) = F(\bar x^T+ \bar a^T - \bar x^T)$, and using repeatedly the $L_1$-smoothness of $F$, we have
    \begin{align*}
        F(\bar a^T)
        &\leq F(\bar{x}^T +  \bar a^{T-1} - \bar x^{T-1}) + \alpha_T + \frac{L_1\lambda_T^2\left\|a^T - x^T\right\|^2}{2}\\
        \hdots &\leq F\left(\bar{x}^T \right) + \sum_{t\in [T]}\alpha_t + \frac{L_1}{2}\sum_{t\in[T]}\lambda_t^2\|a^t - x^t\|^2.
    \end{align*}
    Clearly $(\alpha_t)_{t\in[T]}$ is a martingale difference sequence. Hence, $\mathbb{E}\big[\sum_{t\in[T]} \alpha_t\big] = 0$ 
    and consequently $\mathbb{E}[F(\bar{a}^T) - F(\bar{x}^T)] \leq \frac{L_1}{2}\sum_{t=1}^T\lambda_t^2\mathbb{E}\left[\left\|a^t - x^t\right\|^2\right]$. Since $-F$ is also $L_1$-smooth, the same idea yields $\mathbb{E}[F(\bar{x}^T) - F(\bar{a}^T)] \leq \frac{L_1}{2}\sum_{t=1}^T\lambda_t^2\mathbb{E}\left[\left\|a^t - x^t\right\|^2\right]$.
\end{proof}

We now instantiate  Lemma~\ref{lem:maurey} for the space $(\mathbb{R}^d,\|\cdot\|_1)$ with $\lambda_i = \frac{1}{T}$, $i \in [T]$. In later sections we will see that $x^t \in \Delta_d$ represents the $t$-th iterate of our algorithm and $a^t \sim P_{x^t}$ a private, unbiased estimator of $x^t$. 

\begin{corollary}[Function value approximation w.r.t.~$\|\cdot\|_1$ over the simplex]\label{cor:maurey_over_simplex} Let $F: \RR^d \to \RR$ be a $L_1$-smooth function w.r.t.~$\|\cdot\|_1$. Let $x^1,...,x^T \in \Delta_d$ be deterministic vectors and $a^1,...,a^T \in \Delta_d$ be random vectors such that $\mathbb{E}[a^t | a^{t-1},..,a^1] = x^t$ for all $t \in [T]$. Then 
$\left|\mathbb{E}\left[F\left(\bar a^T\right) - F\left(\bar x^T\right)\right] \right| \leq \frac{2L_1}{T}$.
\end{corollary}

The above result bounds the bias incurred by evaluating gradients at sampled vertices. More concretely, the following corollaries hold. 

\begin{corollary}[{Gradient bias with second-order-smoothness}]\label{cor:gradient_bias_L2smooth}
Suppose $F: \RR^d \to \RR$ has partial derivates that are $L_2$-smooth w.r.t.~$\|\cdot\|_1$. Let $x^1,...,x^T \in \Delta_d$ be deterministic vectors and $a^1,...,a^T \in \Delta_d$ be random vectors such that $\mathbb{E}[a^t | a^{t-1},..,a^1] = x^t$ for all $t \in [T]$. 
Then $\left\|\mathbb{E}\left[\nabla F\left(\bar a^T\right) - \nabla F\left(\bar x^T\right)\right] \right\|_{\infty} \leq \frac{2L_2}{T}$.
\end{corollary}
\begin{proof}
    This follows directly from applying Corollary \ref{cor:maurey_over_simplex} to each partial derivate: for all $j\in[d]$, we have $\left|\mathbb{E}\left[\nabla_j F\left(\bar a^T\right) - \nabla_j F\left(\bar x^T\right)\right] \right| \leq \frac{2L_2}{T}$.
\end{proof}

\begin{corollary}[{Gradient bias with first-order-smoothness}]\label{cor:gradient_bias_L1smooth}
Suppose $F: \RR^d \to \RR$ is $L_1$-smooth w.r.t.~$\|\cdot\|_1$.Let $x^1,...,x^T$ be a collection of deterministic vectors in $\Delta_d$ and $a^1,...,a^T \in \Delta_d$ be random vectors such that $\mathbb{E}[a^t | a^{t-1},..,a^1] = x^t$ for all $t \in [T]$. 
Then $\left\|\mathbb{E}\left[\nabla F\left(\bar a^T\right) - \nabla F\left(\bar x^T\right)\right] \right\|_{\infty} \leq \frac{4L_1}{\sqrt{T}}$.
\end{corollary} 

\begin{proof}
    Denote $\bar{x} = \frac{1}{T} \sum_{t\in[T]} x^t, \bar{a} = \frac{1}{T} \sum_{t\in[T]} a^t$. By $L_1$-smoothness of $F$, we have for all $j\in[d], r \in \RR^+$: $F(\bar x + r e_j) \leq F(\bar x) + \langle\nabla F(\bar x),re_j\rangle + \frac{L_1r^2}{2}\|e_j\|_1^2$.
    Re-arranging this expression gives $-\langle\nabla F(\bar x),e_j\rangle \leq \frac{F(\bar x) - F(\bar x + r e_j)}{r} + \frac{L_1r}{2}$.
    By $L_1$-smoothness of $-F$, we can obtain $\langle\nabla F(\bar a),e_j\rangle \leq \frac{F(\bar{a} + r e_j) - F(\bar a)}{r} + \frac{L_1r}{2}$. Adding these inequalities and taking expectations we get
    \[ \mathbb{E}[\langle\nabla F(\bar a) - \nabla F(\bar x),e_j\rangle] \leq \frac{\mathbb{E}\left[F(\bar a + r e_j) - F(\bar x + r e_j) + F(\bar x) - F(\bar a)\right]}{r} + L_1r.\]
    By applying Corollary \ref{cor:maurey_over_simplex} to $F(\cdot)$ and $F(\cdot + re_j)$ we obtain 
    \[\mathbb{E}\left[F(\bar a + r e_j) - F(\bar x + r e_j) + F(\bar x) - F(\bar a)\right] \leq 4L_1/T.\] Hence, for all $j\in[d]$, $\mathbb{E}[\nabla_j F(\bar a) -\nabla_j F(\bar x)] = \mathbb{E}[\langle\nabla F(\bar a) -\nabla F(\bar x),e_j\rangle] \leq \frac{4L_1}{rT} + L_1r.$ Analogously we can prove the same bound for $\mathbb{E}[\langle\nabla F(\bar x) -\nabla F(\bar a),e_j\rangle]$. Finally, 
    choosing $r = 2/\sqrt{T}$ gives the result. 
\end{proof}

We will also need to bound the second-moment of the error incurred by evaluating gradients at sampled vertices. Towards this, we first observe that the in-expectation guarantees obtained in Lemma \ref{lem:maurey} can be turned into high-probability guarantees by
controlling the increments of the martingale difference sequence $(\alpha_t)_{t\in [T]}$  
with the bounded differences inequality.

\begin{lemma}[Function value approximation with high probability]\label{lem:maurey_whp}
    Let $(\mathbf{E},\|\cdot\|)$ be a normed space and $F$ be an $\lips$-Lipschitz, $L_1$-smooth function w.r.t.~$\|\cdot\|$. Let $x^1,...,x^T$ be a collection of deterministic vectors, $(\lambda_1,\ldots,\lambda_T) \in \Delta_T$ and $a^1,...,a^T$ be random vectors such that $\mathbb{E}[a^t | a^{t-1},..,a^1] = x^t$. Suppose
    $\|x^t-a^t\| \leq D$ a.s.~for all $t\in [T]$. Define $\bar{x}^T =  \sum_{t = 1}^{T} \lambda_t x^t$ and $\bar{a}^T = \sum_{t = 1}^{T} \lambda_t a^t$. Then, for any $\beta > 0$
    \[ \mathbb{P}\left[ F\left(\bar a^T\right)-F\left(\bar x^T\right) \geq \frac{L_1D^2}{2}\sum_{t=1}^T\lambda_t^2+
    \beta \sqrt{2}L_0D\sqrt{\sum_{t=1}^T\lambda_t^2}\,\right] \leq \exp(-\beta^2).\]
\end{lemma}

\begin{proof}
    Let $\alpha_t = \lambda_t\left\langle \nabla F \left( \sum_{k = t}^T \lambda_k x^k + \sum_{k = 1}^{t-1} \lambda_k a^k \right), a^{t} - x^{t} \right\rangle$ as in the proof of Lemma \ref{lem:maurey}. From that proof we have that $F(\bar{a}^T) - F\left(\bar{x}^T \right) \leq  \sum_{t=1}^T\alpha_t + \frac{L_1D^2}{2}\sum_{t=1}^T\lambda_t^2$. Hence
    \begin{equation}\label{eq:deviation_ineq}
        \mathbb{P}[F(\bar{a}^T) - F\left(\bar{x}^T \right) \ge \alpha] \leq \mathbb{P}\left[\sum_{t=1}^T\alpha_t + \frac{L_1D^2}{2}\sum_{t=1}^T\lambda_t^2 \ge \alpha\right] = \mathbb{P}\left[\sum_{t\in [T]}\alpha_t \ge \alpha'\right],
    \end{equation}
    where $\alpha = \frac{L_1D^2}{2}\sum_{t=1}^T\lambda_t^2 + \alpha'$. Next, we bound the probability on the right-hand side. Note that $(\alpha_t)_{t \in [T]}$ is a martingale difference sequence 
    and $|\alpha_i| \leq L_0D \lambda_i$, so the bounded differences inequality gives $\mathbb{P}\left[\sum_{t=1}^T\alpha_t \ge \beta \sqrt{2}L_0D\sqrt{\sum_{t=1}^T\lambda_t^2}\right] \leq \exp(-\beta^2)$, for any $\beta > 0$. Making $\alpha' = \beta \sqrt{2}L_0D\sqrt{\sum_{t=1}^T\lambda_t^2}$ and replacing back into Equation \ref{eq:deviation_ineq}, we conclude the claimed result.
\end{proof}

We now use Lemma \ref{lem:maurey_whp} to give bounds on the expected second-moment error, with respect to $\|\cdot\|_{\infty}$. 

\begin{lemma}[Second moment of maximum expected error] \label{lem:maurey_max_error}
    Let $(\mathbf{E},\|\cdot\|)$ be a normed space and $F_1,...,F_M$ a collection of $L_0$-Lipschitz and $L_1$-smooth functions w.r.t~$\|\cdot\|$. Let $x^1,\ldots,x^T \in \RR^d$, $(\lambda_1,\ldots,\lambda_T)\in \Delta_T$  and $a^1,...,a^T$ be a collection of random vectors such that $\mathbb{E}[a^t|a^{t-1},\ldots,a^1] = x^t$ and $\|a^t-x^t\|\leq D$ a.s.~for all $t\in [T]$. Let $\bar{a}^T = \sum_{t \in [T]} \lambda_t a^t$ and $\bar x^T=\sum_{t\in[T]}\lambda_t x^t$. Then 
    \[
    \mathbb{E}\left[\max_{j \in [M]} |F_j(\bar a^T) - F_j(\bar x^T)|^2\right] \leq \frac{L_1^2D^4}{2}\left(\sum_{t=1}^T\lambda_t^2\right)^2 + 2L_0^2D^2(4+\log(M))\sum_{t=1}^T\lambda_t^2.\]
\end{lemma}

\begin{proof}
    Define the constants $A = \frac{L_1D^2}{2}\sum_{t=1}^T\lambda_t^2$ and $C=\sqrt{2}L_0D\sqrt{\sum_{t=1}^T\lambda_t^2}$. We will introduce a quantity, $B$, in the following computation, and later optimize over the choice of this $B$. We proceed as follows:
    \begin{align*}
        \mathbb{E}\bigg [\max_{j \in [M]} |F_j(\bar a^T) - F_j(\bar x^T)|^2 &\bigg ] =  \int_{0}^{\infty} \mathbb{P}\left[\max_{j \in [M]} |F_j(\bar a^T) - F_j(\bar x^T)|\ge \beta\right]\beta d\beta\\
        &\hspace{-2.3cm}\leq \int_0^{A+B} \beta d\beta + \sum_{j\in [M]}\int_{A+B}^{\infty} \mathbb{P}\left[|F_j(\bar a^T) - F_j(\bar x^T)|\ge \beta\right]\beta d\beta\\
        &\hspace{-2.3cm}= \frac{(A+B)^2}{2} + \sum_{j \in [M]}\int_{B/C}^{\infty} \mathbb{P}\left[|F_j(\bar{x}) - F_j(\bar{a}^T)|\ge A + C\beta\right](A + C\beta) Cd\beta\\
        &\hspace{-2.3cm}\leq \frac{(A+B)^2}{2} + 2MC \left[C\int_{B/C}^{\infty} \exp\left(-\beta^2\right)\beta d\beta + A \int_{B/C}^{\infty} \exp\left(-\beta^2\right) d\beta\right]\\ 
        &\hspace{-2.3cm}= \frac{(A+B)^2}{2} + 2MC \left[C\exp\left(-(B/C)^2\right) + A\sqrt{\pi} \mathbb{P}[X \ge B/C]\right],
    \end{align*}
    where the first inequality comes from the union bound, the second from Lemma \ref{lem:maurey_whp}, and in the last line, we use density of a normal random variable, $X \sim {\cal N}(0, 1/2)$. From Mill's inequality, we have $\mathbb{P}[X \ge B/C] \leq \frac{C}{B\sqrt{\pi}}\exp\left(-(B/C)^2)\right)$. Hence, 
    \[\mathbb{E}\left[\max_{j \in [M]} |F_j(\bar a^T) - F_j(\bar x^T)|^2\right] \leq \frac{(A+B)^2}{2} + 2MC(C+A)\exp\left(-(B/C)^2\right).\]
    Finally, setting $B = C\sqrt{\log(M)}$ if $M>1$ and $B=C$ if $M=1$ gives
    \[\mathbb{E}\left[\max_{j \in [M]} |F_j(\bar a^T) - F_j(\bar x^T)|^2\right] \leq 2A^2 + (4+\log(M))C^2.\]
\end{proof}

We conclude this section by presenting two corollaries of Lemma \ref{lem:maurey_max_error}.

\begin{corollary}[Gradient approximation error with second-order-smoothness] \label{cor:grad_approx_second_mom_L2}
    Let $F: \RR^d \to \RR$ be a $L_1$-smooth function w.r.t.~$\|\cdot\|_{1}$ with $L_2$-smooth partial derivatives. Let $x^1,...,x^T \in \Delta_d$ be a collection of deterministic vectors and $a^1,...,a^T \in \Delta_d$ be random vectors such that $\mathbb{E}[a^t | a^{t-1},..,a^1] = x^t$ for all $t \in [T]$. Then
    \[
    \mathbb{E}\left\|\nabla F\left(\bar a^T\right) - \nabla F\left(\bar x^T\right)\right\|^2_{\infty} \leq \frac{8L_2^2}{T^2} + \frac{8L_1^2(4+\log(d))}{T}.\]
\end{corollary}
\begin{proof}
    Apply Lemma~\ref{lem:maurey_max_error} with $F_j=\nabla_j F$, $j\in[d]$. 
\end{proof}

\begin{corollary}[Gradient approximation error with first-order-smoothness] \label{cor:grad_approx_second_mom_noL2}
    Let $F: \RR^d \to \RR$ be a function that is $L_0$-Lipschitz and $L_1$-smooth w.r.t.~$\|\cdot\|_{1}$. Let $x^1,...,x^T \in \Delta_d$ be a collection of deterministic vectors and $a^1,...,a^T \in \Delta_d$ be random vectors such that $\mathbb{E}[a^t | a^{t-1},..,a^1] = x^t$ for every $t \in [T]$. Then
    \[
    \mathbb{E}\left\|\nabla F\left(\bar a^T\right) - \nabla F\left(\bar x^T\right)\right\|^2_{\infty} \leq \frac{8\sqrt{2}L_1^2}{T^{3/2}\sqrt{4 + \log(d)}} + \frac{8\sqrt{2}(L_0^2+L_1^2)\sqrt{(4+\log(d))}}{\sqrt{T}}.\]
\end{corollary}
\begin{proof} 
    Similarly to the proof of Corollary \ref{cor:gradient_bias_L1smooth}, using $L_1$-smoothness of $F$,
    \begin{align*}
        \langle\nabla F(\bar a) -\nabla F(\bar x),e_j\rangle &\leq \frac{F(\bar a + r e_j) - F(\bar x + r e_j) + F(\bar x) - F(\bar a)}{r} + L_1r \\
        &\leq \frac{|F(\bar a + r e_j) - F(\bar x + r e_j) + F(\bar x) - F(\bar a)|}{r} + L_1r.
    \end{align*}
    Since $-F$ is also $L_1$-smooth, then 
    \[\langle\nabla F(\bar x) - \nabla F(\bar a),e_j\rangle \leq \frac{|F(\bar a + r e_j) - F(\bar x + r e_j) + F(\bar x) - F(\bar a)|}{r} + L_1r.\]
    We conclude that 
    \begin{align*}
        \left\|\nabla F\left(\bar a\right) - \nabla F\left(\bar x\right)\right\|^2_{\infty} &= \max_{j\in[d]} \left|\nabla_j F\left(\bar a\right) - \nabla_j F\left(\bar x\right)\right|^2 = \max_{j\in[d]} \left|\langle\nabla F(\bar a) - \nabla F(\bar x),e_j\rangle\right|^2\\
        &\hspace{-1cm}\leq 2\max_{j\in[d]}\frac{|F(\bar a + r e_j) - F(\bar x + r e_j) + F(\bar x) - F(\bar a)|^2}{r^2} + L_1^2r^2.
    \end{align*}
    By Lemma \ref{lem:maurey_max_error} with $F_j(\cdot) = F(\cdot + re_j) - F(\cdot)$ and $M = d$, we have
    \[\mathbb{E}\left[
    \max_{j\in[d]}\frac{|F(\bar a + r e_j) - F(\bar x + r e_j) + F(\bar x) - F(\bar a)|^2}{r^2}\right] \leq \frac{32L_1^2}{T^2r^2} + \frac{32L_0^2(4+\log(d))}{Tr^2}.\]
    Putting everything together and taking expectations
    \[\mathbb{E}[ \left\|\nabla F\left(\bar a\right) - \nabla F\left(\bar x\right)\right\|^2_{\infty}] \leq \frac{64L_1^2}{T^2r^2} + \frac{64L_0^2(4+\log(d))}{Tr^2} + 2L_1^2 r^2.\]
    Choosing $r^2 = \sqrt{\frac{32(4+\log(d))}{T}}$ gives the bound.
\end{proof}

\section{Differentially Private Stochastic Saddle Points}\label{sec:DPSSPs}

In this section we state and prove our main results regarding the DP-SSP problem in $\ell_1$ setting. Given a batch $B$ of iid samples from $\D$, we will denote $\nabla F_\D(x,y;B):= \frac{1}{|B|} \sum_{z \in B}\nabla f(x,y;z)$, and similarly for partial derivatives.

\subsection{Private Stochastic Mirror Descent with Sampling from Vertices}\label{subsec:DPSSPs_nobiasred}
Our first algorithm is a (entropic) stochastic mirror descent (SMD) method with sampling from vertices; see Algorithm \ref{Alg:DPSSP_no_bias_reduction}. The following theorem holds.

\begin{algorithm}[t!]
\caption{Private stochastic mirror descent with vertex sampling}\label{Alg:DPSSP_no_bias_reduction}
\begin{algorithmic}[1]
\REQUIRE Dataset $S$, Number of steps $T$, Step-size $\tau$, Number of iterate samples~$K$
\STATE $x^1 = (1/d_x,...,1/d_x)$, $y^1 = (1/d_y,...,1/d_y)$ 
\FOR{$t = 1$ to $T$}
    \STATE $(\hat{x}^t, \hat{y}^t) = \frac{1}{K}\sum_{k \in [K]}(\hat{x}^{t,k}, \hat{y}^{t,k})$, where $(\hat{x}^{t,k}, \hat{y}^{t,k})_{k \in [K]}\overset{\text{iid}}{\sim} P_{x^t}\times P_{y^t}$
    \STATE Let $B^t$ be a batch of $B = n/T$ fresh samples 
    \STATE $g^t_x = \nabla_x F_\D(\hat{x}^t,\hat{y}^t ; B^t)$, $g^t_y = -\nabla_y F_\D(\hat{x}^t,\hat{y}^t ; B^t)$\label{line:gradient_estimators}
    \STATE $x^{t+1}_j \propto x^t_j \exp\left(-\tau g^t_{x,j}\right), \quad\forall j\in [d_x]$ \quad and \quad $y^{t+1}_i \propto y^t_i \exp\left(\tau g^t_{y,i}\right), \quad\forall i\in [d_y]$ \label{line:nonDPx}
\ENDFOR

\RETURN $(\tilde{x}, \tilde{y}) = \frac{1}{T}\sum_{t = 1}^{T} (\tilde{x}^t, \tilde{y}^{t})$, where $(\tilde{x}^t,\tilde{y}^t)_{t\in[T]} \overset{\text{iid}}{\sim} P_{x^t}\times P_{y^t}$
\end{algorithmic}
\end{algorithm}

\begin{theorem}[Guarantees for Algorithm \ref{Alg:DPSSP_no_bias_reduction}]\label{thm:convergence_DPSSP_no_bias_reduction} 
Let $0 < \delta < 1$, $0 < \varepsilon < 8\log(1/\delta)$. Suppose we are solving an SSP problem under the $\ell_1$ setting. Then,  there exists $T, \tau$ and $K$ such that Algorithm \ref{Alg:DPSSP_no_bias_reduction} is $(\varepsilon,\delta)$-DP and its output $(\tilde{x},\tilde{y})$ satisfies
\begin{equation}\label{eqn:DPSSP_l1smooth_noBR}
    \mathbb{E}[\gap(\tilde{x},\tilde{y})] \lesssim  (L_0 + L_1)\left[\sqrt{\frac{\ell}{n}} + \bigg(\frac{\ell^{3/2}\sqrt{\log(1/\delta)}}{n\varepsilon}\right)^{1/3}\bigg].
\end{equation}
If in addition $f(\cdot,\cdot;z)$ has $L_2$-smooth partial derivatives, then there exists a different setting for $T, \tau$ and $K$ such that Algorithm \ref{Alg:DPSSP_no_bias_reduction} is $(\varepsilon,\delta)$-DP and its output satisfies
\begin{equation}\label{eqn:DPSSP_l2smooth_noBR}
    \mathbb{E}[\gap(\tilde{x},\tilde{y})]\lesssim (L_0 + L_1 + L_2)  \bigg[\sqrt{\frac{\ell}{n}} + \left(\frac{\ell^{3/2}\sqrt{\log(1/\delta)}}{n\varepsilon}\right)^{2/5}\bigg].
\end{equation}
Finally, if $L_2 = 0$, then there exists another setting for $T, \tau$ and $K$ such that Algorithm \ref{Alg:DPSSP_no_bias_reduction} is $(\varepsilon,\delta)$-DP and 
returns $(\tilde{x},\tilde{y})$ with
\begin{equation}\label{eqn:DPSSP_quadratic_noBR}
    \mathbb{E}[\gap(\tilde{x},\tilde{y})] \lesssim  (L_0 + L_1) \bigg [\sqrt{\frac{\ell}{n}} + \left(\frac{\ell^{3/2}\sqrt{\log(1/\delta)}}{n\varepsilon}\right)^{1/2}\bigg].
\end{equation}

\end{theorem}

\begin{remark}
    The duality gap bound for $L_2=0$ (i.e.~quadratic objective) in Theorem \ref{thm:convergence_DPSSP_no_bias_reduction} is optimal up to constant factors 
    in the case where $\log(d_x)\simeq \log(d_y)$. This follows from the lower bound in \cite{lower_bound_DP_query_answering} with a resampling argument from \cite{Bassily:2019}.
    Moreover, the proof of 
    this rate only requires $L_2\lesssim (L_0 + L_1)\Big\{\sqrt{\ell/n} + \sqrt{\ell^{3/2}\sqrt{\log(1/\delta)}/(n\varepsilon)}\Big\}$, so the optimal rates hold for more general nearly quadratic objectives. 
\end{remark}

\subsubsection{Privacy Analysis of Algorithm \ref{Alg:DPSSP_no_bias_reduction}} 

We prove that the algorithm is DP. More concretely, we provide a condition on the stepsize that suffices to get privacy.  

\begin{proposition} \label{prop:privacy_PMWU_Maurey}

Let $0 < \delta < 1$, $0 < \varepsilon < 8\log(1/\delta)$. Then, Algorithm \ref{Alg:DPSSP_no_bias_reduction} is $(\varepsilon,\delta)$-DP if $\tau\leq \frac{B\varepsilon}{16L_0\sqrt{T(K+1)\log(1/\delta)}}$.
\end{proposition}

The proof relies on the privacy guarantees of the exponential mechanism and on the advanced composition property of DP. For completeness, we include them below. 

\begin{theorem}[Exponential mechanism (Theorem 3.10 in \cite{dwork2014DPfoundations})]\label{thm:exp_mech}
     Consider a (score) function $s:{\cal Z}^n\times J\mapsto\mathbb{R}$, where $J=\{e_i\}_{i\in[d]}$.  
    Define the sensitivity of $s$ as $\Delta_s := \max_{e \in J}\max_{S, S' \in \Z^n \text{ differing in }\leq 1\text{ element}}|s(S,e) - s(S',e)|$. Then, sampling an object $e$ from $J$ according to $\mathbb{P}(e = e_i) \propto \exp\left(\frac{\varepsilon s(S,e_i)}{2\Delta_s}\right)$ is $\varepsilon$-DP. As a consequence, sampling $e$ according to  $\mathbb{P}(e = e_i) \propto \exp\left(s(S,e_i)\right)$ is $2\Delta_s-$DP. 
\end{theorem}

\begin{theorem}[Fully Adaptive Composition \cite{Whitehouse:2023}]\label{thm:fully_adaptive_composition} Suppose $(\A_n)_{n\ge 1}$ is a sequence of algorithms such that, for any $n\ge 1$, $\A_n$ is $(\varepsilon_n, \delta_n)$-DP with respect to a dataset $S$ when conditioned on $\A_{1:n-1}$. Let $\varepsilon \ge 0$ and $\delta = \delta^{\prime} + \delta^{\prime\prime}$ be the target privacy parameters such that $\delta^{\prime}, \delta^{\prime\prime} > 0$. Consider the function $N: \RR^{\infty}_{\ge 0} \times \RR^{\infty}_{\ge 0} \to \mathbb{N}$ such that $N((\varepsilon_n)_{n\ge1}, (\delta_n)_{n\ge1})$ is defined as
\[\inf \left\{n: \varepsilon < \sqrt{2\log\left(\frac{1}{\delta^{\prime}}\right) \sum_{m \leq n+1} \varepsilon_m^2} + \frac{1}{2}\sum_{\sum_{m\leq n+1}}\varepsilon_m^2 \text{ or } \delta^{\prime\prime} < \sum_{m \leq n +1} \delta_m \right\}.\]
Then, the algorithm $\A_{1:N(S)}$ is $(\varepsilon, \delta)$-DP, where $N(S) = N((\varepsilon_n(S))_{n\ge1}, (\delta_n(S))_{n\ge1})$.
\end{theorem}

The following corollary follows from the theorem above, by noting that $\varepsilon' \leq \frac{\varepsilon}{2\sqrt{2T\log(1/\delta)}}$ and $0 \leq \varepsilon \leq 8\log(1/\delta)$ imply $ \varepsilon \ge \varepsilon'^2 T/2 + \varepsilon' \sqrt{2T\log (1/\delta)}$. 
    
\begin{corollary}[Advanced composition of pure DP mechanisms]\label{cor:pureDPadvanced_compostion}
    For $\varepsilon' > 0$. the \(T\)-fold adaptive composition of \(\varepsilon'\)-DP mechanisms satisfies \((\epsilon, \delta)\)-DP for any $\delta \in (0,1)$ as long as $\varepsilon' \leq \frac{\varepsilon}{2\sqrt{2T\log(1/\delta)}}$ and $0 \leq \varepsilon \leq 8\log(1/\delta)$. 
\end{corollary}

\begin{proof}[Proof of Proposition \ref{prop:privacy_PMWU_Maurey}]
We start proving that for all $t\in [T], k \in [K]$, $\tilde{x}^{t+1},\tilde{y}^{t+1}$ and $ \hat{x}^{t+1, k},\hat{y}^{t+1, k}$ are private given the previous private iterates. 
First, notice that $\tilde{x}^{1},\tilde{y}^{1}$ and $ \hat{x}^{1, k},\hat{y}^{1, k}$ are trivially private because they don't depend on the dataset $S$.
Now we consider $\tilde{x}^{t+1}, \hat{x}^{t+1, k}$. Unrolling the $x$ update,
$\textstyle x^{t+1}_j \propto \exp\left\{-\tau \left(\sum_{i=1}^{t}g^{i}_{x,j}\right)\right\}$,
where $g^{t}_{x,j}$ denotes the $j$-th coordinate of $g^{t}_{x}$ for every $t \in [T]$. Hence, $\tilde{x}^{t+1}$ and $\hat{x}^{t+1, k}$ are vertices of the simplex $\xsimplex$, and each of them is choosen with the exponential mechanism with score function $s^t_x(S, j) = -\tau \left(\sum_{i=1}^{t}g^{i}_{x,j}\right)$. The privacy guarantee of the exponential mechanism in Theorem \ref{thm:exp_mech} implies that $\tilde{x}^{t+1}$ and $\hat{x}^{t+1, k}$ are $2\Delta(s^t_x)$-DP, where $\Delta(s^t_x)$ is the sensitivity of $s^t_x$. Using the same idea, we argue that $\tilde{y}^{t+1}$ and $\hat{y}^{t+1,k}$ are all $2\Delta(s^t_y)$-DP, where $\Delta(s^t_y)$ is the sensitivity of $\textstyle s^t_y(S, j) = -\tau \left(\sum_{i=1}^{t}g^{i}_{y,j}\right)$.

By $L_0$-Lipschitzness of $f$ we have that   $\max_{j \in [d]}\left|s^t_x(S, j) - s^t_x(S',j)\right| \leq \frac{2 L_0 \tau }{B}$. 
Hence, every sampled vertex $\tilde{x}^{t+1},\hat{x}^{t+1,k}$ is $\frac{4L_0\tau}{B}$-DP, conditionally on the previously sampled vertices. Analogously,  $\tilde{y}^{t+1}$ and $\hat{y}^{t+1,k}$ are $\frac{4L_0\tau}{B}$-DP, conditionally the on previously sampled vertices.  
By Advanced Composition of DP (Corollary \ref{cor:pureDPadvanced_compostion}), it suffices that each sampled vertex is 
$\frac{\varepsilon}{2\sqrt{4T(K+1)\log(1/\delta)}}$-DP in order to make the entire Algorithm \ref{Alg:DPSSP_no_bias_reduction} $(\varepsilon,\delta)$-DP. Then, for the algorithm to be $(\varepsilon,\delta)$-DP, it suffices that $\frac{4L_0\tau}{B}\leq \frac{\varepsilon}{2\sqrt{4T(K+1)\log(1/\delta)}}$.
\end{proof}

\subsubsection{Convergence Analysis of Algorithm \ref{Alg:DPSSP_no_bias_reduction}}\label{subsec:convergenceofSMD}
For the accuracy analysis, note that Algorithm \ref{Alg:DPSSP_no_bias_reduction} produces both a private sequence $(\tilde{x}^t, \tilde{y}^t)_{t\in [T]}$ and a non-private one $(x^t, y^t)_{t\in [T]}$. 
We will bound the expected duality gap evaluated at the private iterates' average. To do this, we first show a gap bound on the non-private iterates' average, and then we quantify the discrepancy between the duality gap evaluated at the private and non-private iterates' average.
\begin{lemma}[Duality Gap at non-private iterates' average (Implicit in \cite{JNLS:2009})]\label{lem:deterministic_gap_bound}
Let $(x^t, y^t)$ be the (non-private) iterates of Algorithm \ref{Alg:DPSSP_no_bias_reduction} generated in line \ref{line:nonDPx}, and let $\bar{x} = \frac{1}{T}\sum_{t\in[T]} x^t$ and $\bar{y} = \frac{1}{T}\sum_{t\in[T]} y^t$. Furthermore, let $g^t_x$ and $g^t_y$ be the estimators constructed in line \ref{line:gradient_estimators} of the algorithm and $g^t = (g^t_x,g^t_y)$. Denote $\widetilde\Delta_t = \Phi_\D(x^t,y^t) - g^t$. Then, for some sequences $(w^t)_{t \in [T]} \subset \Delta_{x}$ and $(v^t)_{t \in [T]} \subset \Delta_{y}$ that are predictable w.r.t.~the natural filtration, the following inequality holds
\begin{align*}
    T\gap(\bar{x}, \bar{y}) &\leq \frac{2\ell}{\tau} + \tau \sum_{t = 1}^T (\|g^t\|_{\infty}^2 + \|\widetilde\Delta_t\|_{\infty}^2) + \sum_{t = 1}^T \langle \widetilde\Delta_t, (x^t, y^t) - (w^{t},v^{t})\rangle.
\end{align*}

\end{lemma}

For the proof of this lemma, we will use the following auxiliary result. This result is well known in the SSP literature, and follows directly from regret bounds for the (online) mirror-descent method.  Similarly to other works, we use this result in order to circumvent uniform convergence to bound the gap function.

\begin{lemma}[Corollary 2 in \cite{juditsky_VI}]\label{lem:nemirovski_trick_used}
    Let $\xi^1,\xi^2,...$ be a sequence in $\mathbb{R}^d$ and $\X$ be a compact, convex set. Let $\psi: \RR^d \to \RR$ be a $1$-strongly convex w.r.t~$\|\cdot\|$ and $D_{\psi}(x,y) = \psi(x) - \psi(y) -\langle \nabla \psi(y), x-y\rangle$. Define $w^1 \in \X$, $w^{t+1} = \operatorname{argmin}_{x \in \mathcal{X}} D_{\psi}(w^{t}, x) + \langle \xi^{t}, x\rangle$. Then, for all $w\in \mathcal{X}$ and $T \in \mathbb{N}$: 
    $\sum_{t \in [T]} \langle \xi^t, w^{t} - w\rangle \leq D_{\psi}(w^1,w) + \frac{1}{2}\sum_{t \in [T]} \|\xi^t\|^2_*$, where $\|\cdot\|_*$ is the dual norm of $\|\cdot\|$.
\end{lemma}

\begin{proof}[Proof of Lemma  \ref{lem:deterministic_gap_bound}]
    By convexity-concavity of the loss and the convergence rate of Stochastic Mirror Descent in the simplex setup \cite{JNLS:2009, juditsky_VI}:
    \begin{align*}
        &T[F_\D(\bar{x}, v) - F_\D(w, \bar{y})] \leq \sum_{t\in[T]} F_\D(x^t, v) - F_\D(w, y^t) \\
        &\leq \sum_{t\in[T]} \langle \Phi_\D(x^t,y^t), (x^t,y^t)   - (w,v)\rangle \leq \frac{\ell}{\tau} + \tau \sum_{t \in [T]} \|g^t\|_{\infty}^2 + \sum_{t = 1}^T \langle \widetilde\Delta_t, (x^t, y^t) - (w,v)\rangle,
    \end{align*}
    for all $ w \in \xsimplex, v \in \ysimplex$.
    Lemma \ref{lem:nemirovski_trick_used} with $\xi^{t} = \tau(\nabla_x  F_\D(x^t, y^t) - g^t_x)$, $\psi$ the entropy function, $\|\cdot\| = \|\cdot\|_1$ and $w^1 = (1/d_x,...,1/d_x)$ gives 
    \[\sum_{t \in [T]} \langle \tau(\nabla_x  F_\D(x^t, y^t) - g^t_x), w^{t} - w\rangle \leq \log(d_x) + \frac{1}{2}\sum_{t \in [T]} \tau^{2}\|\nabla_x  F_\D(x^t, y^t) - g^t_x\|_{\infty}^2.\]
    Using this result and denoting $\widetilde\Delta_{t,x} = \nabla_x  F_\D(x^t, y^t) - g^t_x$ we can conclude that    
    \[
    \sum_{t \in [T]} \langle \widetilde\Delta_{t,x}, x^t - w\rangle
    \leq \sum_{t \in [T]} \langle \widetilde\Delta_{t,x}, x^t - w^{t}\rangle + \frac{\log(d_x)}{\tau} + \frac{\tau}{2}\sum_{t \in [T]} \|\nabla_x  F_\D(x^t, y^t) - g^t_x\|_{\infty}^2.
    \]
    An analogous bound for the $y$ iterate holds.
    Hence, for any $w \in \xsimplex, v \in \ysimplex$,
    \begin{align*}
    T[F_\D(\bar{x}, v) - F_\D(w, \bar{y})] &\leq \frac{2\ell}{\tau} + \tau \sum_{t = 1}^T (\|g^t\|_{\infty}^2 + \|\widetilde\Delta_t\|_{\infty}^2)+ \sum_{t = 1}^T \langle \widetilde\Delta_t, (x^t, y^t) - (w^{t},v^{t})\rangle.
\end{align*}
Maximizing over $w \in \xsimplex, v \in \ysimplex$ finishes the proof. 
\end{proof}

Now we quantify the discrepancy between the gap of the private (sampled) and non-private iterates. The duality gap is defined in terms of function values, but since it involves a maximum (and the maximum of smooth functions is not necessarily smooth) our sparsification results do not directly apply. We handle this issue with a smoothing argument that leads to the following lemma.

\begin{lemma}\label{lem:nonprivate_2_private_gap}
    Let $(\tilde{x}, \tilde{y})$ be the output of Algorithm \ref{Alg:DPSSP_no_bias_reduction} and define $(\bar{x}, \bar{y})$ as in Lemma \ref{lem:deterministic_gap_bound}. Then, $\mathbb{E}[\gap(\tilde{x},\tilde{y}) - \gap(\bar{x}, \bar{y})] \leq \frac{4L_1}{T} + \frac{2L_1 \sqrt{\ell}}{\sqrt{T}}$.
\end{lemma}

\begin{proof}[Proof of Lemma \ref{lem:nonprivate_2_private_gap}]
    Let $F_{\lambda}(x) := \max_{y \in \ysimplex} [F_\D(x,y) + \lambda \phi(y)]$ where $\phi$ is the negative entropy. Since $F_\D(x,y)$ is $L_1$-smooth in $(x,y)$ and concave in $y$, and $\lambda \phi(y)$ is $\lambda$-strongly convex in $y$, by Proposition 1 in \cite{generalized_fenchel-young_losses},  $F_{\lambda}$ is $\big(L_1 + \frac{L_1^2}{\lambda}\big)$-smooth w.r.t~$\|\cdot\|_{1}$. Hence by Corollary \ref{cor:maurey_over_simplex} applied to $F_{\lambda}$, we have $\mathbb{E}[F_{\lambda}(\tilde{x}) - F_{\lambda}(\bar{x})] \leq \frac{2L_1}{T} + \frac{2L_1^2}{\lambda T}$. In addition, it is easy to see that  $\max_{v \in \ysimplex} F_\D(x,v) \leq F_{\lambda}(x) \leq \lambda \log(d_y) + \max_{v \in \ysimplex} F_\D(x,v)$ for all $x \in \xsimplex$, since $0 \leq \phi(y) \leq \log(d_y)$. Therefore,
    \begin{multline*}
        \mathbb{E}\big[\max_{v \in \ysimplex} F_\D(\tilde{x},v) - \max_{v \in \ysimplex} F_\D(\bar{x},v)\big] = \mathbb{E}\big[\max_{v \in \ysimplex} F_\D(\tilde{x},v) - F_{\lambda}(\tilde{x})\big] + \mathbb{E}[F_{\lambda}(\tilde{x}) - F_{\lambda}(\bar{x})]\\
        + \mathbb{E}\big[F_{\lambda}(\bar{x}) - \max_{v \in \ysimplex} F_\D(\bar{x},v)\big] \leq \frac{2L_1}{T} + \frac{2L_1^2}{\lambda T} + \lambda \log(d_y). 
    \end{multline*}
    Analogously, $\mathbb{E}\big[\min_{w \in \xsimplex} F_\D(w, \bar{y}) - \min_{w \in \xsimplex} F_\D(w, \tilde{y})\big] \leq \frac{2L_1}{T} + \frac{2L_1^2}{\lambda T} + \lambda \log(d_x)$. Using both inequalities, we get   
    \begin{align*}
        \mathbb{E}[\operatorname{Gap}(\tilde{x},\tilde{y})] &= \mathbb{E}[\max_{v \in \ysimplex} F_\D(\tilde{x},v) - \min_{w \in \xsimplex} F_\D(w, \tilde{y})]\\ 
        &= \mathbb{E}[\max_{v \in \ysimplex} F_\D(\tilde{x},v) - \max_{v \in \ysimplex} F_\D(\bar{x},v)] +  \mathbb{E}[\max_{v \in \ysimplex} F_\D(\bar{x},v) - \min_{w \in \xsimplex} F_\D(w, \bar{y})]\\
        &+ \mathbb{E}[\min_{w \in \xsimplex} F_\D(w, \bar{y}) - \min_{w \in \xsimplex} F_\D(w, \tilde{y})]\\
        &\leq \frac{2L_1}{T} + \frac{2L_1^2}{\lambda T} + \lambda \log(d_y) + \mathbb{E}[\operatorname{Gap}(\bar{x}, \bar{y})] + \frac{2L_1}{T} + \frac{2L_1^2}{\lambda T} + \lambda \log(d_x)\\
        &= \frac{4L_1}{T} + \frac{4L_1^2}{\lambda T} + \lambda \ell + \mathbb{E}[\operatorname{Gap}(\bar{x}, \bar{y})].
    \end{align*}
    Choosing $\lambda = L_1\frac{2}{\sqrt{T\ell}}$ gives the bound. 
\end{proof}

With these results, we are ready to prove the convergence guarantees. 
\begin{proof}[Convergence of Algorithm \ref{Alg:DPSSP_no_bias_reduction}]
Following the notation in Algorithm \ref{Alg:DPSSP_no_bias_reduction}, denote $(\tilde{x}, \tilde{y}) = \frac{1}{T}\sum_{t = 1}^{T} (\tilde{x}^t, \tilde{y}^{t})$ and  $(\bar{x}, \bar{y}) = \frac{1}{T}\sum_{t = 1}^{T} (x^t, y^{t})$. 
Using Lemma \ref{lem:deterministic_gap_bound} and the fact that $\|g^t\|_{\infty}^2 + \|\Phi_\D(x^t,y^t) - g^t\|_{\infty}^2 \leq 5L_0^2$ by $L_0$-Lipschitzness of the objective, we obtain $\gap(\bar{x},\bar{y}) \leq \frac{2\ell}{\tau T} + 5\tau L_0^2 + \frac{1}{T}\sum_{t = 1}^T \langle \Phi_\D(x^t,y^t) - g^t, (x^t, y^t) - (w^{t},v^{t})\rangle$. 
Since $w^t, v^t$ are deterministic when conditioning on the randomness up to iteration $t$, taking expectations on both sides of the previous inequality leads to
\begin{align}
    \mathbb{E}[\gap(\bar{x},\bar{y})] 
    &\lesssim \frac{2\ell}{\tau T} + 5\tau L_0^2 + \frac{1}{T}\sum_{t = 1}^T\mathbb{E}\|\mathbb{E}[\Phi_\D(x^t,y^t) - g^t \mid \F_t]\|_{\infty}.\label{eqn:convergence_bound_noBR}
\end{align}

First, we prove Equation \eqref{eqn:DPSSP_l1smooth_noBR}, where we only assume first-order-smoothness. From  Corollary \ref{cor:gradient_bias_L1smooth}, $\|\mathbb{E}[\Phi_\D(x^t,y^t) - g^t \mid \F_t]\|_{\infty} \leq \frac{4L_1}{\sqrt{K}}$. In addition, from Lemma \ref{lem:nonprivate_2_private_gap} it follows that $\mathbb{E}[\gap(\tilde{x},\tilde{y}) - \gap(\bar{x},\bar{y})] \leq 8L_1\sqrt{\ell}/\sqrt{T}$. These bounds lead to 
$
    \mathbb{E}[\gap(\tilde{x},\tilde{y})] \leq  \frac{8L_1\sqrt{\ell}}{\sqrt{T}} + \frac{2\ell}{\tau T} + 5\tau L_0^2 + \frac{2\sqrt{2}L_1}{\sqrt{K}}
$.
Setting $T \eqsim \min\Big\{n, \frac{(n\varepsilon)^{2/3}}{\log(1/\delta)^{1/3}}\Big\}$, $\tau \eqsim \min\Big\{\frac{1}{L_0}\sqrt{\frac{\ell}{T}}, \frac{n\varepsilon}{L_0T\sqrt{TK\log(1/\delta)}}\Big\}$ and $K\eqsim \frac{T}{\ell}$ gives the claimed bound on $\mathbb{E}[\gap(\tilde{x},\tilde{y})]$. This setting is obtained by optimizing the bound over $T,\tau,K$, taking into account the constraint over $\tau$ that Proposition \ref{prop:privacy_PMWU_Maurey} requires for the privacy.

Now, we prove Equation \eqref{eqn:DPSSP_l2smooth_noBR}. Note that $|\mathbb{E}[\nabla_{x,j} F_\D(x^t,y^t) - g^t_{x,j} \mid \F_t]| \leq \frac{2L_2}{K}$ by Corollary \ref{cor:gradient_bias_L2smooth} and the second-order-smoothness assumption. Similarly, we have $|\mathbb{E}[-\nabla_{y,i} F_\D(x^t,y^t) - g^t_{y,i} \mid \F_t]| \leq \frac{2L_2}{K}$ for all $i \in [d_y]$. We conclude that $\|\mathbb{E}[\Phi_\D(x^t,y^t) - g^t \mid \F_t]\|_{\infty} \leq \frac{2L_2}{K}$. Moreover, by Lemma \ref{lem:nonprivate_2_private_gap}, $\mathbb{E}[\gap(\tilde{x},\tilde{y}) - \gap(\bar{x},\bar{y})] \leq 8L_1\sqrt{\ell}/\sqrt{T}$. Hence $\mathbb{E}[\gap(\tilde{x},\tilde{y})] \lesssim L_1\frac{\sqrt{\ell}}{\sqrt{T}} +  \frac{\ell}{\tau T} + \tau L_0^2 + \frac{L_2}{K}$.
Setting $T \eqsim \min\Big\{n, \frac{(n\varepsilon)^{4/5}}{\ell^{1/5}\log(1/\delta)^{2/5}}\Big\}$, $\tau \eqsim \min\Big\{\frac{1}{L_0}\sqrt{\frac{\ell}{T}}, \frac{n\varepsilon}{L_0T\sqrt{TK\log(1/\delta)}}\Big\}$ and $K\eqsim \sqrt{\frac{T}{\ell}}$ proves \eqref{eqn:DPSSP_l2smooth_noBR}.

Finally, we prove Equation \eqref{eqn:DPSSP_quadratic_noBR}. Since $L_2 = 0$, we can safely set $K=1$. Optimizing over $T,\tau$ gives the bound.
\end{proof}

\subsection{A constant probability duality gap bound via Private Stochastic Mirror Descent with Bias-Reduced Sampling from Vertices}\label{subsec:DPSSPs_biasred}
The limitation that prevents Algorithm \ref{Alg:DPSSP_no_bias_reduction} from attaining optimal rates is the bias of the stochastic gradient oracles. This is a consequence of sampling from vertices, and can only be mitigated by increasing the sampling rate. However, this eats up the privacy budget. In this section, we propose circumventing this limitation by combining our gradient oracles with a bias-reduction method; see Algorithm \ref{Alg:DPSSP_bias_reduction}. We first show how that this procedure leads to improved rates, but only with constant probability. In Section \ref{sec:high_prob} we show how to obtain a high probability bound of roughly the same order.

\begin{theorem}[Guarantees for Algorithm \ref{Alg:DPSSP_bias_reduction}]\label{thm:convergence_bias_reduction}
    Let $0 < \delta < 1$, $0 < \varepsilon < 8\log(1/\delta)$. Suppose the objective $f(\cdot,\cdot;z)$ in our SSP problem in the $\ell_1$ setting is 
    second-order-smooth.
    Let $(\tilde{x}, \tilde{y})$ be the output of Algorithm \ref{Alg:DPSSP_bias_reduction} with $M = \log(\sqrt{U})$, stopping parameter $U = \min\Big\{\frac{n\varepsilon\sqrt{C}}{\sqrt{4\cdot48\cdot81\ell\log(1/\delta)}L_0},\frac{n}{2}\Big\}$, $\alpha = \Big(\frac{2\varepsilon^2}{48\cdot 81\log(1/\delta)(\tau L_0)^2n}\Big)^{1/3}$ and step-size $\tau = \sqrt{\frac{\ell}{CU}}$, where $C = L_0^2+L_2^2+\ell M L_1^2$. Then, $(\tilde{x}, \tilde{y})$ is $(\varepsilon, \delta)$-DP and if $U\geq 4$, 
    \[\mathbb{P}\left[\gap(\tilde{x}, \tilde{y}) \lesssim \log(n)\bigg [\sqrt{C}\sqrt{\frac{\ell}{n}} + \sqrt{L_0}C^{1/4}\left(\frac{\ell^{3/2}\sqrt{\log(1/\delta)}}{n\varepsilon}\right)^{1/2}\bigg]\right] \ge 0.99.\]
\end{theorem}

Up to poly-logarithmic factors in $n,d_x,d_y$, the rate above is indeed optimal. Before proving this result, we will explain Algorithm \ref{Alg:DPSSP_bias_reduction} in more detail. 

\paragraph{Bias-Reduced Gradient Oracle} We say a random variable $N$ is distributed as a \textit{truncated geometric}  
with parameters $(p, M)$ if $\mathbb{P}[N = k] = p^{k}\mathbbm{1}_{(k \in \{0,1,...,M\})}/C_M$, where the normalizing constant is $p \leq C_M \leq 1$. We denote it by $\mbox{TG}(p, M)$. To construct a bias-reduced estimator of the gradient at $(x,y)$, we start by sampling $N \sim \mbox{TG}(0.5, M)$. Then, we construct a batch $B$ of $\max\{1,2^{N^t}/\alpha\}$ fresh samples, where $\alpha$ is a parameter to be choosen later. This batch $B$ combined with the gradient estimates (see Algorithm \ref{Alg:BiasReducedGradient} and note similar algorithms in \cite{Blanchet:2015}) reduces the bias while preserving privacy and accuracy (see Thm.~\ref{thm:convergence_bias_reduction}). To do so, we control the bias and second moment of the gradient estimator.

\begin{algorithm}[t!]
\caption{BiasReducedGradient($x, y, N, B$)}\label{Alg:BiasReducedGradient}
\begin{algorithmic}
\STATE Sample $(\hat{x}^i,\hat{y}^i)_{i \in [2^{N+1}]}\overset{\text{iid}}{\sim}  P_{x}\times P_{y}$\label{line:many_samplings}
\STATE $(\bar x_+, \bar y_+) = \frac{1}{2^{N+1}} \sum_{i \in [2^{N+1}]} (\hat{x}^i, \hat{y}^i)$ ; $(\bar x_-, \bar y_{-}) = \frac{1}{2^{N}} \sum_{i \in [2^{N}]} (\hat{x}^i, \hat{y}^i)$ 
\STATE $g_x = C_M2^N(\nabla_x f(\bar x_+, \bar y_+ ; B) - \nabla_x f(\bar x_-, \bar y_- ; B)) + \nabla_x f(\hat{x}^1, \hat{y}^1 ; B)$
\STATE $g_y = - C_M2^N(\nabla_y f(\bar x_+, \bar y_+ ; B) - \nabla_y f(\bar x_-, \bar y_- ; B)) - \nabla_y f(\hat{x}^1, \hat{y}^1 ; B)$
\RETURN $g_x, g_y$
\end{algorithmic}
\end{algorithm}

\begin{lemma}[Bias and Second Moment of Algorithm \ref{Alg:BiasReducedGradient}]\label{lem:BiasReducedGradient_properties} 
Suppose the objective of the SSP problem in the $\ell_1$ setting is second-order-smooth.  Let $x\in \xsimplex, y\in \ysimplex$, $N \sim \mbox{TG}(0.5, M)$ and $B \sim \D^N$. Let $(g_x,g_y) =\text{BiasReducedGradient}(x, y, N, B)$. Then, 
\begin{align*}
\|\mathbb{E}[\Phi_{\D}(x,y) - (g_x,g_y)]\|_{\infty} &\leq \frac{2L_2}{2^M},\\
\text{and} \quad \mathbb{E}[\|(g_x,g_y)\|_{\infty}^2] &\lesssim 
L_0^2+L_2^2+L_1^2M[\log(d_x)+\log(d_y)].
\end{align*}
\end{lemma}

Consequently, the bias scales inversely proportional to $2^M$, where $M$ is {\em largest possible realization of $N$}. Hence the expected bias is roughly same as if we were sampling $2^M$ times while we actually sample only $2^N$ times. This leads to a more favorable privacy-utility tradeoff. Another consequence is the order of the second moment, $M\ell$. Here $M$ will be choosen as a logarithmic quantity (i.e., $M \lesssim \log(\min\{n\varepsilon, n\})$).  Hence, the method allows us to improve the bias at a only a logarithmic cost in the second moment bound. Note, however, a disadvantage of this gradient oracle is we can only control its second moment for second-order-smooth objective functions. 

\begin{proof}[Proof of Lemma \ref{lem:BiasReducedGradient_properties}]
Following the notation in Algorithm \ref{Alg:BiasReducedGradient}, we have 
\begin{align*}
    &\mathbb{E}[g_x] = \textstyle \mathbb{E}\left[\sum_{k = 0}^M \nabla_x f(\bar x_{+}, \bar y_{+} ; B) - \nabla_x f(\bar x_{-}, \bar y_{-} ; B) + \frac{2^{-k}}{C_M}\nabla_x f(\hat x^1, \hat y^1 ; B) \mid N = k\right]\\
    &= \textstyle \mathbb{E}\left[\sum_{k = 0}^M \big(\nabla_x F_\D(\bar x_{+}, \bar y_{+}) - \nabla_x F_\D(\bar x_{-}, \bar y_{-})\big) + \frac{2^{-k}}{C_M}\nabla_x F_\D(\hat x^1, \hat y^1) \mid N = k\right]\\
    &= \textstyle\mathbb{E}\left[\nabla_x F_\D(\bar x_{+}, \bar y_{+}) \mid N = M\right] - \mathbb{E}\left[\nabla_x F_\D(x_{-}, y_{-}) \mid N = 0\right] + \mathbb{E}\left[\nabla_x F_\D(\hat x^1, \hat y^1)\right]\\
    &= \textstyle\mathbb{E}\left[\nabla_x F_\D(\bar x_{+}, \bar y_{+}) \mid N = M\right],
\end{align*}
where the first equality follows from the law of total expectation, the second from the fact that $(\bar{x}_{+},\bar{y}_{+})$,$ (\bar{x}_{-},\bar{y}_{-})$, $(\hat x^1, \hat y^1)$ and $B$ are conditionally independent given $N$, the third from $\mathbb{E}\left[\nabla_x F_\D(\bar x_+, \bar y_+)\mid N = k\right] = \mathbb{E}\left[\nabla_x F_\D(\bar x_-, \bar y_-)\mid N = k+1\right]$ and the last one from $\mathbb{E}\left[\nabla_x F_\D(x_{-}, y_{-}) \mid N = 0\right] = \mathbb{E}\left[\nabla_x F_\D(\hat x^1, \hat y^1)\right]$. 
Finally, we get $\|\mathbb{E}[g_x] - \nabla_x F_\D(x,y)\|_{\infty} = \max_{j\in[d]} |\mathbb{E}[g_{x,j}] - \nabla_{x,j} F_\D(x,y)| \leq \frac{2L_2}{2^M}$ by Corollary \ref{cor:maurey_over_simplex}. Clearly, an analogous bound holds for the gradient with respect to $y$. Writing $\|\mathbb{E}[(g_x,g_y)] -\Phi_\D(x,y)\|_{\infty} = \max\{\|\mathbb{E}[g_x] - \nabla_x F_\D(x,y)\|_{\infty} , \|\mathbb{E}[g_y] - \nabla_y F_\D(x,y)\|_{\infty}\}$ gives the bias bound.

We now proceed with the second moment bound. Then, $\mathbb{E}\|g_x\|_{\infty}^2$ is equal to
\begin{align*}
&\textstyle \sum_{k=0}^M \mathbb{E}\|C_M2^k[\nabla_x f(\bar x_+, \bar y_+ ; B_k) - \nabla_x f(\bar x_-, \bar y_- ; B)] + \nabla_x f(\hat x^1, \hat y^1 ; B)\mid N = k\|_{\infty}^2\frac{2^{-k}}{C_M}\\
&\textstyle \leq 2\left[\sum_{k=0}^M \!C_M2^k\mathbb{E}\|\nabla_x f(\bar x_+, \bar y_+ ; B) \!-\! \nabla_x f(\bar x_-, \bar y_- ; B)\mid N = k\|_{\infty}^2 \!+\! \frac{2^{-k}}{C_M}L_0^2\right]\\
&\textstyle\leq 2L_0^2 + 4C_M\sum_{k=0}^M 2^k\mathbb{E}\|\nabla_x f(\bar x_+, \bar y_+ ; B) - \nabla_x f(x,y; B)\mid N = k \|_{\infty}^2\\
&\textstyle\quad+ 4C_M\sum_{k=0}^M\mathbb{E}2^k\|\nabla_x f(x,y; B) - \nabla_x f(\bar x_-, \bar y_- ; B)\mid N = k\|_{\infty}^2.
\end{align*}
By Corollary 
\ref{cor:grad_approx_second_mom_L2},  
$\mathbb{E}\|\nabla_x f(\bar x_+, \bar y_+ ; B) - \nabla_x f(x,y;B)\mid N = k\|_{\infty}^2 \lesssim \frac{L_2^2}{2^{2(k+1)}}+\frac{\log(d_x)L_1^2}{2^{k+1}}$,
and similarly 
$\mathbb{E}\|\nabla_x f(x,y;B) - \nabla_x f(\bar x_-, \bar y_- ; B) \mid N = k\|_{\infty}^2 \lesssim \frac{L_2^2}{2^{2k}}+\frac{\log(d_x)L_1^2}{2^k}$. 
Hence,  $\mathbb{E}\|g_x\|_{\infty}^2 \lesssim L_0^2 + L_2^2 + M\log(d_x)L_1^2$. Analogously,  $\mathbb{E}\|g_y\|_{\infty}^2 \lesssim L_0^2 + L_2^2 + M\log(d_y)L_1^2$. 
\end{proof}

\paragraph{Bias-reduced Algorithm} 
We propose Algorithm \ref{Alg:DPSSP_bias_reduction}, which is essentially Algorithm~\ref{Alg:DPSSP_no_bias_reduction} combined with the bias-reduced gradient oracle. The main challenges encountered are the oracles becoming heavy-tailed (moments are exponentially larger than their  expectation), and, due to a randomization of batch sizes, it is difficult to control the privacy budget. To handle the heavy-tailed nature of our estimators, we resort to efficiency estimates from Markov's inequality which only certify success with constant probability. For handling the privacy budget challenge, we use the Fully Adaptive Composition Theorem of DP \cite{Whitehouse:2023}, which allows us to set up the mechanisms and their privacy parameters in a random (but predictable) way. 

\begin{algorithm}[t!]
\caption{Private stochastic mirror descent with bias-reduced vertex sampling}\label{Alg:DPSSP_bias_reduction}
\begin{algorithmic}
\REQUIRE {\small Dataset $S$, Step-size $\tau$, Stopping parameter $U \ge 1$, TG parameter $M$, Batch parameter $\alpha$.} 
\STATE $x^1 = (1/d_x,...,1/d_x)$, $y^1 = (1/d_y,...,1/d_y), t = 1, N^1\sim \mbox{TG}(0.5, M)$ 
\WHILE{$\sum_{i \in [t-1]} 2^{N^i}\leq U - 2^M$}\label{line:stopping_condition}
    \STATE Let $B^t$ be a fresh batch of $\max\{1,2^{N^t}/\alpha\}$ samples. 
    \STATE $(g^t_x, g^t_y) = \operatorname{BiasReducedGradient}(x^t, y^t, N^t, B^t)$
    \STATE $x^{t+1}_j \propto x^t_j \exp\left(-\tau g^t_{x,j}\right), \quad\forall j\in [d_x]$ \quad and \quad $y^{t+1}_i \propto y^t_i \exp\left(\tau g^t_{y,i}\right), \quad\forall i\in [d_y]$
    \STATE $N^{t+1}\sim \mbox{TG}(0.5, M)$ \quad \text{and} set $t = t+1$
\ENDWHILE
\RETURN $(\tilde{x}, \tilde{y}) = \frac{1}{t}\sum_{i = 1}^{t} (\tilde{x}^{i}, \tilde{y}^{i})$ where $(\tilde{x}^i,\tilde{y}^i)_{i\in[t]} \stackrel{iid}{\sim} P_{x^i}\times P_{y^i}$\label{line:private_average_iterate}
\end{algorithmic}
\end{algorithm}

Next, we will prove Theorem \ref{thm:convergence_bias_reduction}. It will be useful to define the stopping time $\T = \sup \big\{ T\in \mathbb{N} : 
\sum_{t=1}^{T}2^{N_t}\leq U - 2^{M}\big\}.$
We note that Algorithm \ref{Alg:DPSSP_bias_reduction} performs $\T + 1$ optimization steps (see condition in Line \ref{line:stopping_condition}). The definition of $\T$ ensures that $\sum_{t=1}^{\T}2^{N_t} \leq U - 2^{M}$, and, since $2^{N_t} \leq 2^M$, we have $\sum_{t=1}^{\T+1}2^{N_t}\leq U$. 
We will first prove the privacy guarantees and then convergence. 

\subsubsection{Privacy Analysis of Algorithm \ref{Alg:DPSSP_bias_reduction}} Since this algorithm runs with random batch sizes and for a random number of steps, we use Fully Adaptive Composition (Theorem \ref{thm:fully_adaptive_composition}) to provide a sufficient condition on $U$ that ensures privacy. 

\begin{proposition}\label{prop:privacy_BR}
    Let $0 < \delta < 1$, $0 < \varepsilon < 8\log(1/\delta)$. If $U\leq \frac{\varepsilon^2}{48\log(1/\delta)(9\tau\alpha L_0)^2}$, then Algorithm \ref{Alg:DPSSP_bias_reduction} is $(\varepsilon, \delta)$-DP. 
\end{proposition}

\begin{proof}
By unraveling the recursion that defines $x^{t+1}$ iterate, it is easy to see that its $j$-th coordinate, $x^{t+1}_j$, is proportional to
\[\textstyle\exp\left\{-\tau \sum_{r=1}^t \Big[ 2^{N_r}\Big( \nabla_x f(\bar x_+^r, \bar y_+^r;B^r)-  \nabla_x f(\bar x_-^r, \bar y_-^r; B^r) \Big) + \nabla_x f(x_0^r, y_0^r ; B^r)\Big]\right\}.\]
Hence sampling a vertex from $\xsimplex$ where the $j$-th vertex is choosen with probability $x^t_j$ corresponds to an application of the exponential mechanism with score function given by $s^t_x = -\tau \sum_{r=1}^t \Big[ 2^{N_r}\Big( \nabla_x f(\bar x_+^r, \bar y_+^r;B^r)-  \nabla_x f(\bar x_-^r, \bar y_-^r; B^r) \Big) + \nabla_x f(x_0^r, y_0^r ; B^r)\Big].$ Next, note that in iteration $t$ of Algorithm \ref{Alg:DPSSP_bias_reduction}, we sample a vertex $\tilde x^t \sim P_{x^t}$ in Line \ref{line:private_average_iterate} and we also sample $2^{N_t+1}$ times from $P_{x^t}$ to construct the bias-reduced gradient (with respect to $x$) estimate in Line \ref{line:many_samplings} of Algorithm \ref{Alg:BiasReducedGradient}. Analogously, sampling vertices from $\ysimplex$ corresponds to the exponential mechanism with score function $s^t_y$, which is defined similarly to $s^t_x$ but using gradients w.r.t~ the $y$ variable.  
Denote by $\Delta(s_x^t)$ and $\Delta(s_y^t)$ the sensitivities of the score functions $s_x^t$ and $s_y^t$, respectively. Recall that exponential mechanism with an score function that has sensitivity $\Delta$ is $2\Delta$-DP. Hence, based on the fully adaptive composition of DP (Theorem \ref{thm:fully_adaptive_composition}), the algorithm is $(\varepsilon, \delta)$-DP if 
$\sqrt{8\log(1/\delta)S} + 2S \leq \varepsilon \text{ where } S = \sum_{t=1}^{\T+1}(2^{N_t+1}+1)(\Delta(s_x^t)^2 + \Delta(s_y^t)^2)$.
Next, we bound $\Delta(s_x^t)$. Fix a pair of neighboring datasets $S$ and $S'$. Denote by $s_x^t$ the score function with $S$ and by $s_x^{\prime t}$ with $S'$. When conditioning on the randomness up to iteration $t$ (recall the batch-size is a random quantity), either $B^r = B^{\prime r}$ for all $r \in [t]$, in which case $s_x^t = s_x^{\prime t}$, or there is exactly one index $r$ such that $B^r$ differ from $B^{\prime r}$ in exactly one datapoint (say $z^*$ and $z^{\prime *}$). In this case,
\begin{align*}
    &|s_x^t - s_x^{\prime t}| =\tau\Big| 2^{N_r}\Big( \nabla_x f(\bar x_+^r, \bar y_+^r;B^r)-  \nabla_x f(\bar x_-^r, \bar y_-^r; B^r) \Big) + \nabla_x f(x_0^r, y_0^r ; B^r)\\
    &\qquad-2^{N_r}\Big( \nabla_x f(\bar x_+^r, \bar y_+^r;B^{\prime r})-  \nabla_x f(\bar x_-^r, \bar y_-^r; B^{\prime r}) \Big) - \nabla_x f(x_0^r, y_0^r ; B^{\prime r})\Big|\\
    &\leq \tau\frac{2^{N_r}\alpha}{2^{N_r}}\Big|\nabla_x f(\bar x_+^r, \bar y_+^r;z^*) - \nabla_x f(\bar x_+^r, \bar y_+^r;z^{\prime *}) + \nabla_x f(\bar x_-^r, \bar y_-^r;z^*) - \nabla_x f(\bar x_-^r, \bar y_-^r;z^{\prime *})\Big|\\
    &\qquad + \frac{\tau\alpha}{2^{N_r}}\Big| \nabla_x f(x_0^r, y_0^r ; z^{*}) - \nabla_x f(x_0^r, y_0^r ; z^{\prime *})\Big| \leq \tau \alpha L_0 (4 + 1/2^{N_r}) \leq 4.5\tau \alpha L_0. 
\end{align*}
Analogously, $\Delta(s_y^t)\leq 4.5\tau \alpha L_0 $. Hence, every sampled vertex (either from $\xsimplex$ or $\ysimplex$), is $9\tau \alpha L_0$-DP. Plugging in the condition for privacy mentioned above, we need  $
\sqrt{2\log(1/\delta)\sum_{t=1}^{\T+1}(4\cdot2^{N_t}+2)(9\tau\alpha L_0)^2} + \frac{1}{2}\sum_{t=1}^{\T+1}(4\cdot2^{N_t}+2)(9\tau\alpha L_0)^2 \leq \varepsilon$.
We further restrict this stopping time by upper bounding each of the above terms by $
\varepsilon/2$. Further noticing that such a bound for the first term implies the same bound for the second term so long as $\varepsilon\leq 8\log(1/\delta)$, we conclude that it suffices
\[ 4\sum_{t=1}^{\T+1}2^{N_t}+2(\T+1) \leq 6\sum_{t=1}^{\T+1}2^{N_t} \leq \frac{\varepsilon^2}{8\log(1/\delta)(9\tau\alpha L_0)^2}. \]
This condition is satisfied when $U \leq \frac{\varepsilon^2}{48\log(1/\delta)(9\tau\alpha L_0)^2}$, since
\begin{align*}
    4\sum_{t=1}^{\T+1}2^{N_t}+2(\T+1) &\leq 6\sum_{t=1}^{\T+1}2^{N_t} \leq 6\sum_{t=1}^{\T}2^{N_t} + 6\cdot 2^M \leq 
    6(U - 2^M) + 6\cdot 2^M = 6U.
\end{align*} 
\end{proof}

\subsubsection{Convergence Analysis of Algorithm \ref{Alg:DPSSP_bias_reduction}} 

\begin{proof}[Proof of Theorem \ref{thm:convergence_bias_reduction}] 
For privacy we need $U\leq \frac{\varepsilon^2}{48\log(1/\delta)(9\tau\alpha L_0)^2}$ (Proposition \ref{prop:privacy_BR}). At the same time, we need $\sum_{t\in [\T+1]} \max\{1,2^{N_t}/\alpha\} \leq n$, to ensure Algorithm \ref{Alg:DPSSP_bias_reduction} does not exceed the sample size. For this condition to hold, it suffices that $\T + 1 \leq n/2$ and $\sum_{t\in[\T+1]} 2^{N_t} \leq n\alpha/2$. This way, we only need $\sum_{t\in[\T+1]} 2^{N_t} \leq \min\{n ,n\alpha/2\}$. Since $\sum_{t\in[\T+1]} 2^{N_t} \leq U$, then for this last condition to hold, it is sufficient that $ U\leq \min\{n ,n\alpha/2\}$. 
It is easy to see that under our parameter setting it holds that $U = \min\left\{\frac{\varepsilon^2}{48\log(1/\delta)(9\tau\alpha L_0)^2}, n\alpha/2, n/2\right\}$, so we don't exceed our privacy budget and we satisfy the sample size constraint. 

Now, let $(\tilde{x}, \tilde{y})$ be the output of the algorithm. We will prove that $\gap(\tilde{x}, \tilde{y}) \leq A/B$ with probability at least $0.99$, for some $A, B$ to be defined. In other words, we will prove that $\PP[\gap(\tilde{x}, \tilde{y}) \ge A/B] \leq 0.01$. Notice that 
\begin{equation}\label{eqn:probability_inequality}
    \PP\Big[\gap(\tilde{x}, \tilde{y}) \ge \frac{A}{B}\Big] 
    \leq \PP[(\T+1)\gap(\tilde{x}, \tilde{y}) \ge A] + \PP\Big[\frac{1}{(\T+1)} \ge \frac{1}{B}\Big].
\end{equation}
Note the second probability in  the right hand side can be bounded as follows:
\[\PP\Big[\frac{1}{\T+1} \ge \frac{1}{B}\Big] =  \PP[\T\leq B-1] \leq \PP\Big[{\textstyle\sum_{t\in[B]}}2^{N_t}> U - 2^M\Big] \leq \frac{B \mathbb{E}[2^{N_1}]}{U - 2^M},\]
where the first inequality holds since $\T\leq B - 1$ implies that $\sum_{t\in[B]} 2^{N_i} > U$, and the second follows from Markov's inequality. To bound the first probability on the right hand side of \eqref{eqn:probability_inequality}, we use Markov's inequality again
\[\PP\left[(\T+1)\gap(\tilde{x}, \tilde{y}) \ge A\right] \leq  \frac{\mathbb{E}\left[(\T+1)(\gap(\tilde{x}, \tilde{y}) - \gap(\bar{x},\bar{y})) + (\T+1)\gap(\bar{x},\bar{y})\right]}{A}.\]

Let us bound  $\mathbb{E}\left[(\T+1)\gap(\bar{x},\bar{y})\right]$. From Lemma \ref{lem:deterministic_gap_bound}, we have 
\[
    (\T+1)\gap(\bar{x}, \bar{y}) \lesssim \frac{\ell}{\tau} + \frac{\tau}{2} \sum_{t = 1}^{\T+1} (L_0^2 + \|g^t_x\|_{\infty}^2 + \|g^t_y\|_{\infty}^2) + \sum_{t = 1}^{\T+1} \langle \widetilde \Delta_t, (x^t, y^t) - (w^{t},v^{t})\rangle,
\]
where $\widetilde \Delta_t = \Phi_\D(x^t, y^t) - (g_x^t, g_y^t)$. Since $\T+1 \leq \sum_{t\in[\T+1]} 2^{N_t} \leq U$, then
\begin{align*}
    \textstyle\mathbb{E}\left[\sum_{t = 1}^{\T+1}   \|g^t_x\|_{\infty}^2\right] &=\textstyle \mathbb{E}\left[\sum_{t = 1}^U   \mathbb{E}[\|g^t_x\|_{\infty}^2\mathbbm{1}_{(\T \ge t - 1)} \mid \F_t]\right] = \mathbb{E}\left[\sum_{t = 1}^U   \mathbb{E}[\|g^t_x\|_{\infty}^2]\mathbbm{1}_{(\T \ge t-1)}\right]\\
    &\lesssim \textstyle\left(L_0^2 + L_2^2 + \log(d_x) ML_1^2\right)\mathbb{E}\left[\sum_{t = 1}^U   \mathbbm{1}_{(\T + 1 \ge t)}\right]\\
    &\lesssim \left(L_0^2 + L_2^2 + \log(d_x) ML_1^2\right) \mathbb{E}[\T+1],
\end{align*}
where $\F_t$ is the sigma algebra generated by all the randomness in the algorithm that generated $x^t$. In the second and third equalities, we used that $\{\T \ge t-1\} = \big\{\sum_{k \in [t-1]} 2^{N_k} \leq U - 2^M\big\}$ is $\F_t$-measurable (since $2^{N_1},...,2^{N_{t-1}}$ are used to generate $x^t$), and hence it can be taken out of the conditional expectation. Next, for the first inequality we used that $2^{N_t}$ is independent of $\F_t$, so we can use Lemma \ref{lem:BiasReducedGradient_properties} to bound the conditional expectation of the second moment. In the last line we used the fact that the expectation can be written as the sum of the tail probabilities. Analogously,  we can obtain $\mathbb{E}\big[\sum_{t = 1}^{\T+1}   \|g^t_y\|_{\infty}^2\big] \lesssim L_0^2 + L_2^2 + \log(d_y) ML_1^2 \mathbb{E}[\T+1]$. In a similar way, we can use the same technique of introducing the indicator of $\T + 1 \ge t$ and then Lemma \ref{lem:BiasReducedGradient_properties} to bound the bias term: 
$\mathbb{E}\left[\sum_{t = 1}^{\T+1} \langle \widetilde \Delta_t, (x^t, y^t) - (w^{t},v^{t})\rangle\right]\lesssim \frac{L_2\mathbb{E}[\T+1]}{2^M}$.
Now, we take expectations on both sides of the expression given by Lemma \ref{lem:deterministic_gap_bound} and then use the bounds we established above for the second moment and bias of the bias-reduced gradient estimators:
\begin{align*}
    \mathbb{E}[(\T+1)\gap(\bar{x}, \bar{y})] 
    \lesssim\frac{\ell}{\tau} + \tau[L_0^2 + L_2^2 + \ell  ML_1^2]\mathbb{E}[\T+1] + \frac{L_2\mathbb{E}[\T+1]}{2^M}.
\end{align*}
To bound $\mathbb{E}[\T+1]$ we note that $\sum_{t=1}^{\T + 1} 2^{N_t}\leq U$, so using Wald's identity \cite{Wald:1944} yields 
\begin{align*}
    U &\ge \mathbb{E}\left[\sum_{t=1}^{\T + 1} 2^{N_t}\right] = \mathbb{E}\left[\sum_{t=1}^{\T + 1} \mathbb{E}[2^{N_t}\mathbbm{1}_{(\T +1 \ge t)}| N_{t-1},...,N_1]\right] = \mathbb{E}\left[\sum_{t=1}^{\T + 1} \mathbb{E}[2^{N_t}]\mathbbm{1}_{(\T + 1\ge t)}\right]  \\
    &=\textstyle \mathbb{E}[2^{N_1}]\mathbb{E}\left[\sum_{t=1}^{\T + 1}\mathbbm{1}_{(\T + 1\ge t)}\right] = \mathbb{E}[2^{N_1}]\mathbb{E}\left[\T+1\right] \eqsim M \mathbb{E}\left[\T+1\right].
\end{align*}
Hence, $\mathbb{E}\left[\T+1\right] \lesssim U/\mathbb{E}[2^{N_1}] \eqsim U/M$. It follows that \vspace{-0.2cm}
\begin{multline*}
    \mathbb{E}[(\T+1)\gap(\bar{x}, \bar{y})] \lesssim \frac{\ell}{\tau} + \tau[L_0^2 + L_2^2 + \ell ML_1^2]\frac{U}{M} + \frac{L_2U}{M2^M}\\
    \lesssim \frac{\ell}{\tau} + \tau[L_0^2 + L_2^2 + \ell M L_1^2]U + \frac{L_2U}{2^M} \lesssim\sqrt{\ell[L_0^2 + L_2^2 + \ell M L_1^2]U}, 
\end{multline*}
where the last step follows from $M = \log_2(\sqrt{U})$ and $\tau =\sqrt{\frac{\ell}{(L_0^2+L_2^2+\ell M L_1^2)U}}$.

Now, let us bound $\mathbb{E}\left[(\T+1)(\gap(\tilde{x}, \tilde{y}) - \gap(\bar{x},\bar{y}))\right]$. We can assume that $\tilde{x}$'s and $\tilde{y}$'s are sampled after the While Loop is over in the Algorithm, thus by Lemma \ref{lem:nonprivate_2_private_gap}\vspace{-0.1cm}
\begin{multline*}
     \mathbb{E}\left[(\T+1)(\gap(\tilde{x}, \tilde{y}) - \gap(\bar{x},\bar{y}))\right] = \mathbb{E}\left[(\T+1)\mathbb{E}_{\tilde{x}, \tilde{y}}\left[\gap(\tilde{x}, \tilde{y}) - \gap(\bar{x},\bar{y})\mid \T\right]\right]\\
    \textstyle\leq \mathbb{E}\left[(\T+1)\left(\frac{4L_1}{\T+1} + \frac{2L_1\sqrt{\ell}}{\sqrt{\T+1}}\right)\right] = 4L_1 + 2L_1\sqrt{\ell}\mathbb{E}\left[\sqrt{\T+1}\right] \lesssim\sqrt{\ell L_1^2\frac{U}{M}},
\end{multline*}
where the second step follows by Jensen's inequality.
Combining the bounds obtained, we conclude that $\PP\left[(\T+1)\gap(\tilde{x}, \tilde{y}) \ge A\right] \lesssim \frac{\sqrt{\ell[L_0^2 + L_2^2 + (\ell) L_1^2]U}}{A}$. 

Plugging all the bounds in Equation \ref{eqn:probability_inequality}, $\PP\left[\gap(\tilde{x}, \tilde{y}) \ge \frac{A}{B}\right] \lesssim \frac{\sqrt{\ell[L_0^2 + L_2^2 + \ell M L_1^2]U}}{A}+\frac{B M}{U - \sqrt{U}}$.
Since $U \ge 4$, then $U - \sqrt{U} \ge U/2$. Choosing $A \eqsim \sqrt{\ell[L_0^2 + L_2^2 + \ell M L_1^2]U}$, $B \eqsim U/M$ and using $M\lesssim \log(n)$, gives the claimed bound. 
\end{proof}

\subsection{A high probability duality gap bound via Private Stochastic Mirror Descent with Bias-Reduced Sampling from Vertices}\label{sec:high_prob}
In this section, we will prove that Theorem \ref{thm:convergence_bias_reduction} can be turned into a high probability result. This is achieved with Algorithm \ref{alg:boosting}. The key idea is as follows. It is easy to see that by running Algorithm \ref{Alg:DPSSP_bias_reduction} multiple times on independent datasets of size of order $O(n/\log(1/\beta))$ (for any $\beta\in(0,1)$), we can ensure that with probability at least $1-\beta$, one of the output pairs has a duality gap of the same order as the gap obtained in Theorem \ref{thm:convergence_bias_reduction}, up to $\operatorname{polylog}(1/\beta)$ factors. The problem is how to find the pair with small duality gap among all the candidate pairs. The first idea would be to run a DP selection procedure, such as the exponential mechanism. This does not work directly, since the duality gap is not computable. Hence, in order to run the exponential mechanism, we need to create a score function that (1) resembles the duality gap and (2) is computable from samples and we can control its sensitivity. For each candidate pair, we propose a two step procedure to approximate the duality gap while satisfying the conditions (1) and (2). For a fixed pair $x,y$, we first use the fact that the duality gap is defined as a sum of maxima of concave functions, so we can use the Private Variance Reduced Frank-Wolfe from \cite{asi2021private} to find an approximate maximizer, say $(\tilde{x}, \tilde{y})$. The second step is to approximate the stochastic objective by samples. This way, we make two approximations: $\gap(x,y) \approx F_\D(\tilde{x},y) - F_\D(x,\tilde{y}) \approx F_{\tilde S}(\tilde{x},y) - F_{\tilde S}(x,\tilde{y})$, where $F_{\tilde S}(x,y) = \frac{1}{|{\tilde S}|} \sum_{z\in {\tilde S}} f(x,y;z)$ and ${\tilde S}$ is independent of $x,y,\tilde{x}, \tilde{y}$. The error in the first approximation can be controlled with the convergence guarantees of Private Variance Reduced Frank-Wolfe while the second can be controlled by Hoeffding's Inequality under the assumption that $f$ is bounded. Furthermore, $F_{\tilde S}(\tilde{x},y) - F_{\tilde S}(x,\tilde{y})$ is computable from samples. Now that we have approximated the duality gap by a computable quantity, we can run the exponential mechanism with this quantity as score function, which will select a candidate pair that approximately has the smallest duality gap among all the candidates. A detailed procedure can be found in Algorithm \ref{alg:boosting}, and our result is in Theorem \ref{thm:hp_result}. 

\begin{theorem}\label{thm:hp_result}
    Let $0 < \delta < 1$, $0 < \varepsilon < 8\log(1/\delta)$. Suppose we are solving the DPSSP problem in the $\ell_1$ setting where the objective $f(\cdot,\cdot;z): \Delta_x \times \Delta_y \to [-B,B]$ is convex-concave and second-order-smooth for all $z \in \Z$. Let $(x_{i^*}, y_{i^*})$ be the output of Algorithm \ref{alg:boosting}. Then, $(x_{i^*}, y_{i^*})$ is $(\varepsilon, \delta)$-DP and for any $\beta > 0$, with probability at least $1-\beta$, $\gap(x_{i^*}, y_{i^*})$ is bounded above (up to $\operatorname{polylog}(n, 1/\beta)$ and constant factors) by $\sqrt{\frac{\ell}{n}}  + \frac{\ell}{\sqrt{n\varepsilon}} + \frac{B}{\sqrt{n}} + \frac{B}{n\varepsilon}$.
\end{theorem}

\begin{algorithm}
\caption{Boosted private stochastic mirror descent with bias-reduced vertex sampling}
\label{alg:boosting}
\begin{algorithmic}[1]
\STATE \textbf{Input:} Dataset \( S \), parameters $I,J$,
, privacy parameters $\varepsilon, \delta$

\STATE Let $\A^{DPSSP}$ be Algorithm \ref{Alg:DPSSP_bias_reduction} and $\A^{DPSCO}$ be Algorithm 3 from \cite{asi2021private}.
\STATE Partition \( S \) into 4 equal parts \( S_1, S_2, S_3, S_4 \)

\STATE Partition \( S_1 \) into \( I \) equal parts: \( S_1^1, S_1^2, \dots, S_1^I \)
\STATE For each \( i \in [I] \), compute \( (x_i, y_i) = \mathcal{A}^{DPSSP}(S_1^i, \varepsilon, \delta) \)\label{step:candidate_pairs}

\STATE Partition \( S_2 \) and $S_3$ into \( IJ \) equal parts: \( S_2^{11}, S_2^{12}, \dots, S_2^{IJ}; S_3^{11}, S_3^{12}, \dots, S_3^{IJ} \)\label{step:candidate_maximizers}
\STATE For each \( i \in [I], j \in [J] \), compute \( x_{ij} = \mathcal{A}^{DPSCO}(S_2^{ij},\varepsilon, \delta, F_\D(\cdot, y_i)) \) 
\STATE For each \( i \in [I], j \in [J] \), compute \( y_{ij} = \mathcal{A}^{DPSCO}(S_3^{ij},\varepsilon, \delta, -F_\D(x_i, \cdot))) \)  

\STATE For each \( i \in [I] \), define \( G_i(S_4) = \max_{j \in [J]} F_{S_4}(x_i, y_{ij}) +  \max_{j \in [J]} -F_{S_4}(x_{ij}, y_{i})\)

\STATE Sample  
$i^*$ such that $\forall i \in [I]$, $\PP[i^* = i] \propto \exp\left(-n\varepsilon G_i(S_4)/[32B]\right)$\label{step:exp_mech_hpalgorithm}

\STATE \textbf{Return:} $(x_{i^*}, y_{i^*})$
\end{algorithmic}
\end{algorithm}

We first prove the privacy guarantee. 

\begin{proof}[Privacy of Algorithm \ref{alg:boosting}]
    $\A^{DPSSP}$ and $\A^{DPSCO}$ are $(\varepsilon, \delta)$-DP algorithms. Since we run them on disjoint parts of $S$, steps \ref{step:candidate_pairs}  and \ref{step:candidate_maximizers} are $(\varepsilon, \delta)$-DP by parallel composition of DP. The last step is an exponential mechanism with score function $G_i(S_4)$. Conditioned on $(x_i,y_i), x_{ij}, y_{ij}$, $i\in[I], j\in [J]$, the sensitivity of $G_i(S_4)$ is $4B/|S_4|$. Since $|S_4| = n/4$, we get that the sensitivity is $16B/n$. Theorem \ref{thm:exp_mech} gives that step \ref{step:exp_mech_hpalgorithm} is $(\varepsilon,0)$-DP. We conclude using parallel composition one more time that the whole algorithm is $(\varepsilon, \delta)$-DP. 
\end{proof}

To prove the accuracy guarantee we will need two results.

\begin{theorem}[Theorem 7 in \cite{asi2021private}]\label{thm:DPSCOell_1}
     Consider DPSCO in the $\ell_1$ setting as in Section \ref{sec:DPSCO}. Then, Private Variance Reduced Stochastic Frank-Wolfe (Algorithm 3 in \cite{asi2021private}) returns $\tilde{x}$ such that $\mathbb{E}[F_\D(\tilde{x}) - \min_{x\in \Delta_x} F_\D(x)] \lesssim C(d_x, n)$ where  
     \begin{equation}\label{eq:dpsco_error}
         C(d,n) := (L_0+L_1)\frac{\sqrt{\log(d)}\log(n)}{\sqrt{n}} + \left(\frac{\sqrt{L_1}L_0\log(1/\delta)\log(d)\log(n)^2}{n\varepsilon}\right)^{2/3}.
     \end{equation}
\end{theorem}

\begin{remark}
Combining the above with Markov's inequality gives that if we run Private Variance Reduced Frank-Wolfe $k$ times on datasets $S_1,...,S_k \overset{iid}{\sim} \D^n$, and get outputs $\tilde{x}_1,...,\tilde{x}_k$, then with probability at least $1-2^k$, $\min_{i \in [k]}\big(F_\D(\tilde{x}_i) - \min_{x\in \Delta_x} F_\D(x)\big) \lesssim 2C(d_x,n)$.
\end{remark}

\begin{theorem}[Theorem 3.11 in \cite{dwork2014DPfoundations}]\label{thm:accuracy_exp_mech}
   Consider the same setting as in Theorem \ref{thm:exp_mech}. Let $j^*$ be the object sampled from $J$ according to $\mathbb{P}(j = e_i) \propto \exp\left(\frac{\varepsilon s(S,i)}{2\Delta_s}\right)$. Then, for any $\alpha > 0$:
   $\PP\left[s(S,j^*) \leq \max_{j\in [J]} s(S,j) - \frac{2\Delta_s(\log(J) + \log(1/\alpha))}{\varepsilon}\right] \leq \alpha$.
\end{theorem}

\begin{proof}[Accuracy of Algorithm \ref{alg:boosting}] Given constants $C_1, C_2,C_3,C_4, C_5$, define the following events using the notation from the Algorithm \ref{alg:boosting}
\[A_{i}^x = \left\{\max_{x \in \Delta_{x}} - F_\D(x, y_i) \lesssim \max_{j \in [J]} -F_\D(x_{ij}, y_i) + C_1\right\},\] 
\[A_{i}^y =  \left\{\max_{y \in \Delta_{y}} F_\D(x_i, y) \lesssim \max_{j \in [J]} F_\D(x_{i}, y_{ij}) + C_2\right\},\]
\[A_{ij}^x = \left\{|F_\D(x_{ij}, y_i) - F_{S_4}(x_{ij}, y_i)| \lesssim C_3\right\}, A_{ij}^y = \left\{|F_\D(x_{ij}, y_i) - F_{S_4}(x_{ij}, y_i)| \lesssim C_3\right\},\]
\[A_{exp} = \left\{G_{i^*}(S_4) \lesssim \min_{i \in [I]}G_{i}(S_4) +   C_4\right\}, A_{BR} = \left\{\min_{i \in [I]} \gap(x_i,y_i) \lesssim C_5\right\}.\]
By Theorem \ref{thm:DPSCOell_1}, picking $C_1 = C(d_x, n/[4IJ]), C_2 = C(d_y, n/[4IJ])$, with $C(d,n)$ as defined in Equation \eqref{eq:dpsco_error}, gives $\PP[A_i^x] \ge 1 - 2^{-J}, \PP[A_i^y] \ge 1 - 2^{-J}$. Furthermore, since $|f(x,y;z)| \leq B$, Hoeffding's inequality gives $\PP[A_{ij}^x] \ge 1 - \alpha, \PP[A_{ij}^y] \ge 1-\alpha$ for $C_3 \eqsim B\sqrt{\frac{\log(1/\alpha)}{n}}$. By Theorem \ref{thm:accuracy_exp_mech}, if we pick $C_4 \eqsim \frac{B(\log(I) + \log(1/\gamma))}{n\varepsilon}$, then $\PP[A_{exp}] \ge 1-\gamma$. Finally, by Theorem \ref{thm:convergence_bias_reduction}, letting $C = L_0^2+L_2^2+\ell M L_1^2$ we have that for $C_5 \eqsim \log(n)\bigg [\sqrt{C}\sqrt{\frac{I\ell}{n}} + \sqrt{L_0}C^{1/4}\left(\frac{I\ell^{3/2}\sqrt{\log(1/\delta)}}{n\varepsilon}\right)^{1/2}\bigg]$, $\PP(A_{BR}) \ge 1-2^{-I}$. Let $E$ be the intersection of all the events above. Then, with the proposed choices for the constants $C_1,...,C_5$ it holds 
\[\PP[E] \ge 1 - (2^{-I} + \gamma + 2\alpha IJ + 2I2^{-J}),\]
and furthermore, choosing 
\[I = \log_2(4/\beta), \gamma = \beta/4, J = \log_2(8I/\beta), \alpha = \beta/(8IJ) \]
gives that $\PP[E] \ge 1-\beta$ for any $\beta \in (0,1)$. Finally, note that under $E$
\begin{align*}
    \gap(x_{i^*}, y_{i^*}) &= \max_{y \in \Delta_{y}} F_\D(x_{i^*}, y) + \max_{x \in \Delta_{x}}- F_\D(x, y_{i^*})\\
    &\leq \max_{j \in [J]} F_\D(x_{i^*}, y_{i^*j}) + \max_{j \in [J]}- F_\D(x_{i^*j}, y_{i^*}) + C_1 + C_2\\
    &\leq \max_{j \in [J]} F_{S_4}(x_{i^*}, y_{i^*j}) + \max_{j \in [J]}- F_{S_4}(x_{i^*j}, y_{i^*}) + C_1+ C_2 + 2C_3\\
    &= G_{i^*}(S_4) + C_1 + C_2 + 2C_3\\
    &\leq \min_{i \in [I]}G_{i}(S_4) + C_1 + C_2 + 2C_3 + C_4 \\
    &= \min_{i \in [I]}\left\{\max_{j \in [J]} F_{S_4}(x_i, y_{ij}) +  \max_{j \in [J]} -F_{S_4}(x_{ij}, y_{i})\right\} + C_1 + C_2 + 2C_3 + C_4 \\
    &\leq \min_{i \in [I]}\left\{\max_{j \in [J]} F_{\D}(x_i, y_{ij}) +  \max_{j \in [J]} -F_{\D}(x_{ij}, y_{i})\right\} + C_1 + C_2 + 4C_3 + C_4\\
    &\leq \min_{i \in [I]}\left\{\max_{y \in \Delta_{y}} F_{\D}(x_i, y) +  \max_{x \in \Delta_{x}} -F_{\D}(x, y_{i})\right\} + C_1 + C_2 + 4C_3 + C_4\\
    &= \min_{i\in[I]} \gap(x_i, y_i)+ C_1 + C_2 + 4C_3 + C_4\\
    &\leq C_1 + C_2 + 4C_3 + C_4 + C_5.
\end{align*}
Using the definition of $E$ and our choices for $C_1,C_2,C_3,C_4,C_5$, we conclude that with probability at least $1-\beta$, 
\begin{align*}
    &\gap(x_{i^*}, y_{i^*}) \lesssim \frac{(L_0+L_1)\sqrt{\ell\log_2(8\log_2(4/\beta)/\beta)}\log(n)}{\sqrt{n}} \\
    &+ \left(\frac{\sqrt{L_1}L_0\log(1/\delta)\ell\log(n)^2\log_2(8\log_2(4/\beta)/\beta)}{n\varepsilon}\right)^{2/3}\\
    &+ B\sqrt{\frac{\log([8\log_2(4/\beta)\log_2(8\log_2(4/\beta)/\beta))]/\beta)}{n}} + \frac{B(\log(\log_2(4/\beta)) + \log(4/\beta))}{n\varepsilon}\\
    &+ \log(n)\bigg [\sqrt{C}\sqrt{\frac{\log_2(4/\beta)\ell}{n}} + \sqrt{L_0}C^{1/4}\left(\frac{\log_2(4/\beta)\ell^{3/2}\sqrt{\log(1/\delta)}}{n\varepsilon}\right)^{1/2}\bigg].
\end{align*}
This bound is of order
$\sqrt{\frac{\ell}{n}}  + \frac{\ell}{\sqrt{n\varepsilon}} + \frac{B}{\sqrt{n}} + \frac{B}{n\varepsilon}$ up to $\operatorname{polylog}(n, 1/\beta)$ factors.
\end{proof}

\begin{algorithm}[t!]
\caption{Private stochastic mirror descent with vertex-sampling for DPSCO}\label{Alg:DPSCO}
\begin{algorithmic}[1]
\REQUIRE Dataset $S = \{z^1,...,z^n\}$, Total number of steps $T$, Step-size $\tau$, Sampling parameter $K$, Round length $q$, Initial point $x^1 = (1/d_x,...,1/d_x)$
\FOR{$t = 1$ to $T$}
    \STATE $w^t = \frac{(t-1)w^{t-1} + x^t}{t}$ 
    \IF{$t \leq q$ or $t = 0$ mod $q$} 
        \STATE $\hat{w}^t = \frac{1}{K}\sum_{k \in [K]}\hat{w}^{t,k}$, where $(\hat{w}^{t,i})_{i \in [K]}\overset{\text{iid}}{\sim} P_{w^t}$
    \ELSE
        \STATE $\hat{w}^t = \hat{w}^{t-1}$
    \ENDIF 
    \STATE $g^t = \nabla f(\hat{w}^{t}; B^t)$, where $B^t$ is a batch of $B = n/T$ fresh samples.
    \STATE $x^{t+1}_j \propto x^t_j \exp\left(-\tau g^t_j\right), \quad\forall j\in [d_x]$
\ENDFOR
\RETURN $\hat{w}^T = \frac{1}{K}\sum_{k \in [K]}\hat{w}^{T,k}$, where $(\hat{w}^{T,i})_{i \in [K]}\overset{\text{iid}}{\sim} P_{w^T}$
\end{algorithmic}
\end{algorithm}

\section{Differentially Private Stochastic Convex Optimization}\label{sec:DPSCO}

We show how our techniques can be used for DP-SCO in the $\ell_1$ setting. The setup is similar to DP-SSP in the $\ell_1$ setting: consider an unkown distribution $\D$ over a sample space $\Z$. We have access to $\D$ only through a sample $S = \{z^1,...,z^n\} \overset{\text{iid}}{\sim} \D$. Consider a loss function $f : \xsimplex \times \Z \to \RR$, where $\xsimplex$ is the standard simplex in $\RR^{d_x}$. We assume that for all $z \in \Z$, the function $f(\cdot, z)$ is $L_1$-smooth and $L_0$-Lipschitz w.r.t~$\|\cdot\|_1$ over $\xsimplex$. Let $F_\D(x) = \mathbb{E}_{z \sim \D}[f(x;z)]$ be the population loss. 
We are interested in solving $\min_{x\in \xsimplex} F_\D(x)$
with $(\varepsilon, \delta)$-DP algorithms. Many ideas used in the previous section for the SSP problem will be useful for here as well. We observe that guarantees for (entropic) stochastic mirror descent can be obtained by applying online-to-batch conversions on exponentiated gradient descent \cite{orabona_online_learning}. 
By contrast, we will next apply {\em anytime online-to-batch conversions} to exponentiated gradient descent \cite{anytime-online-to-batch}.  The version we provide below for this result is slightly different than \cite{anytime-online-to-batch} (see \cite{online_to_batch} for a proof).
\begin{lemma}[Anytime online-to-batch \cite{anytime-online-to-batch,online_to_batch}]\label{lem:anytimeOTB}
    Let $\X \subseteq \RR^{d_x}$ be a convex set. For any sequence $\beta_t \in \RR^+, g^t \in \RR^{d_x}$, suppose an online learner predicts $x^t$ and receives $t$-th loss $\ell_t(x) = \langle g^t, x\rangle$. Define $w^t = \sum_{i \in [t]}\frac{\beta_i x^i}{\sum_{j\in[t]} \beta_j}$ and $\operatorname{Regret}_T(x) = \sum_{t \in [T]} \langle g_t, x_t - x\rangle$. Then, for any convex, differentiable $F$ over $\X$ and any $x\in \X,$ 
    \[
    F(w^T) - F(x) \leq \frac{1}{\sum_{j\in[T]} \beta_j}\bigg (\operatorname{Regret}_T(x) + \sum_{t \in [T]} \langle \beta_t \nabla F(w^t) - g_t, x_t - x\rangle\bigg).
    \]
\end{lemma}
An advantage of this conversion is that the weighted averages $w^t$ change slowly from one iteration to the next, which is very useful in the smooth case \cite{online_to_batch}. Using Lemma~\ref{lem:anytimeOTB} with weights $\beta_j = 1$ to optimize over $\xsimplex$, the average iterates satisfy $\|w^t - w^{t-1}\|_1 = \|x^t - w^{t-1}\|_1/t\leq 2/t$. 
Consequently, $P_{w^{t-1}}$ and $P_{w^t}$ are progressively closer distributions.  
We exploit this idea in Algorithm~\ref{Alg:DPSCO} by maintaining estimates of the private iterates for a large number of steps, 
aggressively
reducing the number of sampled vertices and improving the privacy. 
Theorem \ref{thm:accuracy_anytime_method_L2smoothness} 
provides a faster rate  compared to the one we would get with the plain online-to-batch conversion. 

\begin{theorem}[Guarantees for Algorithm \ref{Alg:DPSCO}]\label{thm:accuracy_anytime_method_L2smoothness}
Let $0 < \delta < 1$, $0 < \varepsilon < 8\log(1/\delta)$. Let $f(\cdot;z)$ be convex and first-order-smooth. Then, there exists a parameter setting for $T,K,q,\tau$ such that Algorithm \ref{Alg:DPSCO} run with these as inputs returns $\hat{w}^T$ such that $\hat{w}^T$ is $(\varepsilon, \delta)$-DP and
\begin{multline*} \mathbb{E}[F_{\cal D}(\hat w^T)-F_{\cal D}(x)] \lesssim  (L_0 + L_1)\left[\sqrt{\frac{\log(d_x)}{n}} + L_0\frac{\log(d_x)^{7/10}\log(1/\delta)^{1/5}}{(n\varepsilon)^{2/5}}\right]\\
+ L_1\log(\sqrt{n}\log(d_x))\left[\frac{1}{\log(d_x)\sqrt{n}} + \frac{\log(1/\delta)^{1/5}}{\log(d_x)^{4/5}(n\varepsilon)^{2/5}}\right].
\end{multline*}

If $f(\cdot;z)$ is convex and second-order-smooth, Algorithm \ref{Alg:DPSCO} with a different choice for  $T,K,q,\tau$ outputs $\hat{w}^T$ that is $(\varepsilon, \delta)$-DP and satisfies
\begin{multline*} \mathbb{E}[F_{\cal D}(\hat{w}^T)-F_{\cal D}(x)] \lesssim (L_0+L_1)\left[\sqrt{\frac{\log(d_x)}{n}}+\frac{\log(d_x)\log(1/\delta)^{1/4}}{\sqrt{n\varepsilon}}\right]+\frac{L_2}{\sqrt{n}\log(d_x)^{1/4}}\\
+\frac{L_2\log(d_x)^{1/4}\log(1/\delta)^{1/4}}{\sqrt{n\varepsilon}} +L_1\log(\sqrt{n\log(d_x)})\Big[\frac{1}{\sqrt{n}}+\frac{\log(1/\delta)^{1/4}}{\sqrt{n\varepsilon}}\Big].
\end{multline*}
\end{theorem}

\subsection{Privacy Analysis of Algorithm \ref{Alg:DPSCO}} Now we make explicit the advantage of using the anytime online-to-batch approach in terms of privacy. Namely, due to the substantial sampling reduction, our privacy budget degrades much more nicely.

\begin{proposition} \label{thm:privacy_anytime_method}
Let $0 < \delta < 1$, $0 < \varepsilon < 8\log(1/\delta)$. Then, Algorithm \ref{Alg:DPSCO} is $(\varepsilon,\delta)$-DP if $\tau\leq \frac{B\varepsilon}{8L_0\sqrt{2(TK/q+qK)\log(1/\delta)}}$.
\end{proposition}

\begin{proof}
    Note that $w^t = \frac{1}{t} \sum_{i\in[t]} x^i$. Hence, for any $j\in [d_x]$
    \[ w^t_j = \frac{1}{t} \sum_{i\in[t]} x^i_j\propto \sum_{i\in[t]} \exp\left(-\tau \sum_{k \in [i-1]} g^k_j\right) = \exp\left(\log\left(\sum_{i\in[t]} \exp\left(-\tau \sum_{k \in [i]} g^k_j\right)\right)\right).\]
    Denote by $G_j(S) = \big(\sum_{k \in [i-1]} g^k_j\big)_{i\in[t]}$ the vector of the cumulative sum of the gradients that our algorithm produces when it has been fed with dataset $S$. Also, define the log-sum-exp function as $\lse((x_1,...,x_k)) = \log(\sum_{i=1}^k\exp\{x_i\})$. With this new notation, $w^t_j \propto \exp(\lse(-\tau G_j(S)))$. Two well-known properties of the $\lse(x)$ function are that it is convex in its argument and its gradient is the softmax function of $x$. As a consequence of the second property mentioned, it is clear that $\|\nabla \lse(x)\|_1 = 1$ for every $x$. By convexity of $\lse$, we have 
    \begin{align*}
        \lse(-\tau G_j(S))& - \lse(-\tau G_j(S')) \leq \langle \nabla \lse(-\tau G_j(S)), \tau G_j(S') - \tau G_j(S) \rangle\\
        &\leq \tau \|\nabla \lse(-\tau G(S))\|_1\|G_j(S) - G_j(S')\|_{\infty} = \tau \|G_j(S) - G_j(S')\|_{\infty}.
    \end{align*}
    Furthermore,
    $\max_{j\in[d_x]} \tau \|G_j(S) - G_j(S')\|_{\infty} = \tau \max_{j\in[d_x], i \in [t]}  \left|\sum_{k \in [i-1]} (g^k_j - g^{\prime k}_j)\right| \leq \frac{2\tau L_0}{B}$ by $L_0$-Lipschitzness of $f$. Hence sampling vertices from $\xsimplex$ according to the probabilities induced by $w^t$ is equivalent to running the exponential mechanism with score function $s(S,j) = \lse(-\tau G_j(S))$, whose sensitivity is bounded by $\frac{2\tau L_0}{B}$. Since we sample $K$ times during the first $q$ iterations, then every $q$ iterations and then $K$ times to sparsify the last iterate, the overall number of samplings will be $TK/q + qK + K$. Since each of them is $\frac{4\tau L_0}{B}$-DP, then by Advanced Composition of DP (Corollary \ref{cor:pureDPadvanced_compostion}) we only need that 
    $\frac{4\tau L_0}{B} \leq \frac{\varepsilon}{2\sqrt{2(TK/q + (q+1)K)\log(1/\delta)}}$.
\end{proof}

\subsection{Convergence Analysis of Algorithm \ref{Alg:DPSCO}}

\begin{proof}[Proof of Theorem \ref{thm:accuracy_anytime_method_L2smoothness}] Let $x\in\xsimplex$ be a minimizer of $F_\D(\cdot)$. From Lemma \ref{lem:anytimeOTB}, 
    $F_\D(w^T) - F_\D(x) \leq \frac{1}{T} \big[\sum_{t \in [T]} \langle g_t, x_t - x\rangle + \sum_{t \in [T]} \langle \nabla F_\D(w^t) - g^t, x^t-x\rangle\big]$.
    The standard analysis of mirror descent yields, $\sum_{t \in [T]} \langle g_t, x_t - x\rangle \leq \frac{\log(d_x)}{\tau} + \frac{\tau}{2} \sum_{t \in [T]} \|g^t\|_{\infty}^2$ \cite{JNLS:2009}.  
    Using that $\|g^t\|_{\infty} \leq L_0$ and taking expectations, we obtain 
    \[\mathbb{E}[F_\D(w^T) - F_\D(x)] \leq \frac{\log(d_x)}{\tau T} + \frac{\tau L_0^2}{2} + \frac{1}{T}\sum_{t \in [T]} \mathbb{E}\left[\langle \nabla F_\D(w^t) - g^t, x^t-x\rangle\right].\]
    Let us bound  $\mathbb{E}\left[\langle \nabla F_\D(w^t) - g^t, x^t-x\rangle\right]$. For $t\leq q$, we have $g^t \! = \! \nabla f(\hat{w}^t;B^t)$, and 
    \begin{align*}
        &\mathbb{E}[\langle \nabla F_\D(w^t) - \nabla f(\hat{w}^{t} ; B^t), x^t-x \rangle] = \mathbb{E}[\langle \nabla f(w^{t} ; B^t) - \nabla f(\hat{w}^{t} ; B^t), x^t-x\rangle]\\
        &=\mathbb{E}_{x^t}[\langle \mathbb{E}_{\hat{w}^t}[\nabla f(w^{t} ; B^t) - \nabla f(\hat{w}^{t} ; B^t) \mid x^t], x^t-x\rangle]\\
        &\leq 2\|\mathbb{E}_{\hat{w}^{t}}[\nabla f(w^{t} ; B^t) - \nabla f(\hat{w}^{t} ; B^t)]\|_{\infty} \leq \begin{cases}\frac{8L_1}{\sqrt{K}}& \text{for first-order-smooth $f(\cdot;z)$} \\ \frac{4L_2}{K}& \text{for second-order-smooth $f(\cdot;z)$.}\end{cases}
    \end{align*}
    where in the first equality we used that $\nabla f(w^{t} ; B^t)$ is an unbiased estimator of $\nabla F_\D(w^{t})$ and for the second inequality we used Corollaries \ref{cor:gradient_bias_L1smooth} and  \ref{cor:gradient_bias_L2smooth}. 
    Now consider the case $t > q$. First, denote by $t_0 = q \lfloor t/q\rfloor$. That is, $t_0$ was the last time that  $\hat{w}$ was updated by the algorithm before the $t$-th iteration. Hence, $g^t = \nabla f(\hat{w}^{t_0} ; B^t)$. Now, we decompose the error into three terms and bound each of them individually: 
    \begin{align*}
        \langle \nabla F_\D(w^t) - \nabla f(\hat{w}^{t_0} ; B^t), x^t-x\rangle &= \langle \nabla F_\D(w^t) - \nabla f(w^{t} ; B^t), x^t-x\rangle\\ 
        &+ \langle\nabla f(w^{t} ; B^t) - \nabla f(w^{t_0} ; B^t), x^t-x\rangle\\
        &+ \langle \nabla f(w^{t_0} ; B^t) - \nabla f(\hat{w}^{t_0} ; B^t), x^t-x\rangle.
    \end{align*}
    Clearly $\mathbb{E}[\langle \nabla F_\D(w^t) - \nabla f(w^{t} ; B^t), x^t-x\rangle] = 0$, since $\nabla f(w^{t} ; B^t)$ is an unbiased estimator of $\nabla F_\D(w^t)$ conditionally on the randomness up to iteration $t$. For the second term, note that
    \begin{align*}
        &\langle\nabla f(w^{t} ; B^t) - \nabla f(w^{t_0} ; B^t), x^t-x\rangle \leq \|\nabla f(w^{t} ; B^t) - \nabla f(w^{t_0} ; B^t)\|_{\infty}\|x^t-x\|_1\\
        &\leq 2L_1\|w^{t} - w^{t_0}\|_1 \leq  2L_1 \sum_{k = t_0}^{t-1} \|w^{k+1} - w^{k}\|_1 \leq 2L_1 \sum_{k = t_0}^{t-1} \frac{2}{k+1} \leq \frac{4L_1q}{q\lfloor t/q\rfloor + 1},
    \end{align*}
    where the first inequality follows from H\"older, the second from the $L_1$-smoothness of $f$, the third from the triangle inequality for $\|\cdot\|_1$, and the fourth from the stability of the average iterates. Now, we decompose
    \[ \langle \nabla f(w^{t_0} ; B^t) - \nabla f(\hat{w}^{t_0} ; B^t), x^t-x\rangle = A_1+A_2,\]
    with $A_1:=\langle \widetilde{\Delta}^{t_0}, \Tilde{x}^t-x\rangle$ and $A_2:=\langle \widetilde{\Delta}^{t_0}, x^t-\Tilde{x}^t\rangle$, 
    where $\widetilde{\Delta}^{t_0}$ denotes $\nabla f(w^{t_0} ; B^t) - \nabla f(\hat{w}^{t_0} ; B^t)$ and $\tilde{x}^t$ is a ficticious iterate obtained by starting at $x^{t_0}$ and performing updates $\tilde x_j^{t+1}\propto \tilde x_j^t \exp\big(-\tau\nabla_j f(w^{t_0};B^k)\big)$. 
    We start by bounding $A_1$:
    \begin{align*}
        &\mathbb{E}_{\hat{w}^{t_0}}[\langle \widetilde{\Delta}^{t_0}, \Tilde{x}^t-x\rangle] =\langle \mathbb{E}_{\hat{w}^{t_0}}[\nabla f(w^{t_0} ; B^t) - \nabla f(\hat{w}^{t_0} ; B^t)], \Tilde{x}^t-x\rangle\\
        &\leq 2\|\mathbb{E}_{\hat{w}^{t_0}}[\nabla f(w^{t_0} ; B^t) - \nabla f(\hat{w}^{t_0} ; B^t)]\|_{\infty} \leq \begin{cases}\frac{8L_1}{\sqrt{K}}& \text{for first-order-smooth $f(\cdot;z)$} \\ \frac{4L_2}{K}& \text{for second-order-smooth $f(\cdot;z)$,}\end{cases}
    \end{align*}
    where the equality holds since $\tilde{x}^t$, as opposed to $x^t$, does not depend on $\hat{w}^{t_0}$, and the last inequality follows from Corollaries \ref{cor:gradient_bias_L1smooth} and  \ref{cor:gradient_bias_L2smooth}. 
    We bound now $A_2$. Note that $\langle \widetilde{\Delta}^{t_0}, x^t-\Tilde{x}^t\rangle \leq \|\widetilde{\Delta}^{t_0}\|_{\infty} \|x^t - \tilde{x}^t\|_1$.  
    Denote $\hat{G}_j = -\tau\sum_{k = t_0}^{t-1}\nabla_j f(\hat{w}^{t_0};B^k), G_j = -\tau\sum_{k = t_0}^{t-1}\nabla_j f(w^{t_0};B^k)$. We 
    bound $|(x^t - \tilde{x}^t)_j|$ as follows
    \begin{align*}
        (x^t - \tilde{x}^t)_j &=\frac{x^{t_0}_j\exp(\hat{G}_j)}{\sum_{i \in [d_x]}x^{t_0}_i\exp(\hat{G}_i)} - \frac{x^{t_0}_j\exp\left(G_j\right)}{\sum_{i \in [d_x]}x^{t_0}_i\exp\left(G_i\right)}\\
        &= \frac{x^{t_0}_j\exp\left(G_j\right)\exp(\hat{G}_j - G_j)}{\sum_{i \in [d_x]}x^{t_0}_i\exp\left(G_i\right)\exp(\hat{G}_i - G_i)} - \frac{x^{t_0}_j\exp\left(G_j\right)}{\sum_{i \in [d_x]}x^{t_0}_i\exp\left(G_i\right)}\\
        &\leq \frac{x^{t_0}_j\exp\left(G_j\right)}{\sum_{i \in [d_x]}x^{t_0}_i\exp(G_i)}\Big[\exp\big(2\max_{j\in[d_x]} |\hat{G}_j - G_j|\big) - 1\Big],
    \end{align*}
    and similarly $(\tilde{x}^t - x^t)_j \leq \frac{x^{t_0}_j\exp\left(G_j\right)}{\sum_{i \in [d_x]}x^{t_0}_i\exp\left(G_i\right)}\left[2\max_{j\in[d_x]} |\hat{G}_j - G_j|\right]$.
    Combining both inequalities, we obtain $|(x^t - \tilde{x}^t)_j| \leq \frac{x^{t_0}_j\exp\left(G_j\right)}{\sum_{i \in [d_x]}x^{t_0}_i\exp\left(G_i\right)} \psi$, where\\ $\psi := \max\Big\{2\max_{j\in[d_x]} |\hat{G}_j - G_j|, \exp\big(2\max_{j\in[d_x]} |\hat{G}_j - G_j|\big) - 1\Big\}$.
    It follows that $\|x^t - \tilde{x}^t\|_1 = \sum_{j\in [d_x]} |(x^t - \tilde{x}^t)_j| \leq \sum_{j\in[d_x]} \frac{x^{t_0}_j\exp\left(G_j\right)}{\sum_{i \in [d_x]}x^{t_0}_i\exp\left(G_i\right)}\psi = \psi$.
    Next, we use that $e^x - 1 \leq 2x$ for $x\leq 1$. Assume for now $2\max_{j\in[d_x]} |\hat{G}_j - G_j| \leq 1$, so $\psi = 4\max_{j\in[d_x]} |\hat{G}_j - G_j|$. Using the definition of $\hat{G}_j, G_j$, it follows that
    \begin{align*}
        \|x^t - \tilde{x}^t\|_1 \leq \psi &= \textstyle 4\max_{j\in[d_x]} \tau \left|\sum_{k = t_0}^{t-1}[\nabla_j f(\hat{w}^{t_0};B^k) - \nabla_j f(w^{t_0};B^k)]\right|\\
        &\textstyle\leq 4\tau \max_{j\in[d_x]} \sum_{k = t_0}^{t-1}\left|\nabla_j f(\hat{w}^{t_0};B^k) - \nabla_j f(w^{t_0};B^k)\right| \\
        &\textstyle\leq 4\tau \sum_{k = t_0}^{t-1}\left\|\nabla f(\hat{w}^{t_0};B^k) - \nabla f(w^{t_0};B^k)\right\|_{\infty}.
    \end{align*}
    This has two implications. The first is that $\psi \leq 8\tau q L_0$, by $L_0$-Lipschitzness of $f$. This further implies that in order to ensure that  $2\max_{j\in[d_x]} |\hat{G}_j - G_j| \leq 1$ holds, it suffices that $8\tau q L_0\leq 2$, or equivalently $\tau \leq 1/(4L_0q)$. We impose this condition on $\tau$ later when optimizing the convergence rate. The second is the following bound on $A_2$
    \begin{align*}
        &A_2=
        \langle \widetilde{\Delta}^{t_0}, x^t-\Tilde{x}^t\rangle  = \langle \nabla f(w^{t_0} ; B^t) - \nabla f(\hat{w}^{t_0} ; B^t), x^t-\Tilde{x}^t\rangle\\
        &\leq \|\nabla f(w^{t_0} ; B^t) - \nabla f(\hat{w}^{t_0} ; B^t)\|_{\infty} \|x^t - \tilde{x}^t\|_1\\ 
        &\textstyle\leq 4\tau \sum_{k = t_0}^{t-1}\|\nabla f(w^{t_0} ; B^t) - \nabla f(\hat{w}^{t_0} ; B^t)\|_{\infty}\|\nabla f(\hat{w}^{t_0};B^k) - \nabla f(w^{t_0};B^k)\|_{\infty}\\
        &\textstyle\leq 2\tau \sum_{k = t_0}^{t-1}\Big(\|\nabla f(w^{t_0} ; B^t) - \nabla f(\hat{w}^{t_0} ; B^t)\|_{\infty}^2 + \|\nabla f(\hat{w}^{t_0};B^k) - \nabla f(w^{t_0};B^k)\|_{\infty}^2\Big)\\
        &\textstyle\leq 2\tau q\|\nabla f(w^{t_0} ; B^t) - \nabla f(\hat{w}^{t_0} ; B^t)\|_{\infty}^2 + 2\tau\sum_{k = t_0}^{t-1}\|\nabla f(\hat{w}^{t_0};B^k) - \nabla f(w^{t_0};B^k)\|_{\infty}^2\\
        &\lesssim \begin{cases}q\tau\big[\frac{L_1^2}{\sqrt{\log(d_x)}K^{3/2}} + \frac{(L_0^2 + L_1^2)\sqrt{\log(d_x)}}{\sqrt{K}}\big]& \text{for first-order-smooth $f(\cdot;z)$} \\ q\tau \big[\frac{L_2^2}{K^2} + \frac{L_1^2\log(d_x)}{K}\big]& \text{for second-order-smooth $f(\cdot;z)$,}\end{cases}
    \end{align*}
    where we bounded the second moment of the estimation errors with the results in Corollaries \ref{cor:grad_approx_second_mom_L2} and \ref{cor:grad_approx_second_mom_noL2}.
    We conclude that for all $t>q$
    \begin{align*}
        &\mathbb{E}[\langle \nabla F_{\D}(w^t) - \nabla f(\hat{w}^{t_0} ; B^t), x^t-x\rangle] \\
        &= \mathbb{E}\big[\langle \nabla F_{\D}(w^t) - \nabla f(w^{t} ; B^t), x^t-x\rangle\big]
        + \mathbb{E}\big[\langle\nabla f(w^{t} ; B^t) - \nabla f(w^{t_0} ; B^t), x^t-x\rangle\big]\\
        &\quad+ \mathbb{E}\big[\langle \nabla f(w^{t_0} ; B^t) - \nabla f(\hat{w}^{t_0} ; B^t), x^t-x\rangle\big]\\
        &\leq 0 + \frac{4L_1q}{q\lfloor t/q\rfloor + 1} + 2\|\mathbb{E}_{\hat{w}^{t_0}}[\nabla f(w^{t_0} ; B^t) - \nabla f(\hat{w}^{t_0} ; B^t)]\|_{\infty} 
        \\
        & \, + 2\tau q\|\nabla f(w^{t_0} ; B^t) - \nabla f(\hat{w}^{t_0} ; B^t)\|_{\infty}^2 + 2\tau\sum_{k = t_0}^{t-1}\mathbb{E}\left[\|\nabla f(\hat{w}^{t_0};B^k) - \nabla f(w^{t_0};B^k)\|_{\infty}^2\right]\\
        &\lesssim \begin{cases} \frac{L_1}{\sqrt{K}} + q\tau\big[\frac{L_1^2}{\sqrt{\log(d_x)}K^{3/2}} + \frac{(L_0^2 + L_1^2)\sqrt{\log(d_x)}}{\sqrt{K}}\big]& \text{for first-order-smooth $f(\cdot;z)$} \\ \frac{L_2}{K} + q\tau \big[\frac{L_2^2}{K^2} + \frac{L_1^2\log(d_x)}{K}\big]& \text{for second-order-smooth $f(\cdot;z)$.}\end{cases}
    \end{align*} 
    And recall that for $t\leq q$, 
    \[\mathbb{E}[\langle \nabla F_\D(w^t) - \nabla f(\hat{w}^{t} ; B^t), x^t-x\rangle]\leq \begin{cases}\frac{8L_1}{\sqrt{K}}& \text{for first-order-smooth $f(\cdot;z)$} \\ \frac{4L_2}{K}& \text{for second-order-smooth $f(\cdot;z)$.}\end{cases}\]
    
    Finally, we plug this bounds into our excess risk bound. We start considering the case of second-order-smooth functions. We bound $\mathbb{E}[F_\D(w^T) - F_\D(x)] $ by 
    \begin{align}
        &\frac{\log(d_x)}{\tau T} + \frac{\tau L_0^2}{2} + \frac{1}{T}\sum_{t \in [T]} \mathbb{E}\left[\langle \nabla F_\D(w^t) - g^t, x^t-x\rangle\right] \notag \\
        &\textstyle\lesssim \frac{\log(d_x)}{\tau T} + \tau L_0^2 + \frac{1}{T}\sum_{t \in [T]} \Big(\frac{L_2}{K} + \mathbbm{1}_{(t > q)}\Big[\frac{4L_1q}{q\lfloor t/q\rfloor + 1} + q\frac{L_2^2\tau}{K^2} + q\frac{L_1^2\log(d_x)\tau}{K}\Big]\Big) \notag\\
        &\textstyle\lesssim \frac{\log(d_x)}{\tau T} + \tau\Big[ L_0^2 + \frac{L_1^2q\log(d_x)}{K} + \frac{L_2^2q}{K^2} \big]  + \frac{L_2}{K} + \frac{1}{T}\sum_{t = q+1}^T \frac{L_1q}{q\lfloor t/q\rfloor + 1}\notag\\
        &\lesssim \frac{\log(d_x)}{\tau T} + \tau\Big[ L_0^2 + \frac{L_1^2q\log(d_x)}{K} + \frac{L_2^2q}{K^2} \big]  + \frac{L_2}{K} + \frac{L_1 q \log(T/q)}{T},\label{eqn:final_UB_anytime}
    \end{align}
    where the last line follows from $\sum_{t = q+1}^T \left[\frac{1}{q\lfloor t/q\rfloor + 1}\right] \leq q\left(\frac{1}{q} + \frac{1}{2q} + ... + \frac{1}{(T/q)q}\right) \lesssim \log\big(\frac{T}{q}\big)$.
    Now, by Theorem \ref{thm:privacy_anytime_method}, and recalling that the batch size is $B=n/T$, we have a constraint on $\tau>0$ in order to satisfy privacy. We also have the constraint of $\psi \leq 1$, which holds true only when $\tau \leq 1/(8qL_0)$. Optimizing \eqref{eqn:final_UB_anytime} under these constraints, we have that 
    we should choose $\tau$ of order
    \begin{align*}
    \min\bigg\{ &\sqrt{\frac{\log(d_x)}{(L_0^2 + L_1^2q\log(d_x)/K + L_2^2q/K^2)T}},  \frac{1}{L_0q}, \frac{n\varepsilon}{TL_0\sqrt{(TK/q + qK)\log(1/\delta)}}\bigg \}.
    \end{align*} 
    With this choice we obtain 
    \begin{align*}
        \mathbb{E}[F_\D(w^T) - F_\D(x)] &\lesssim \sqrt{\frac{(L_0^2 + L_1^2 q\log(d_x)/K+ L_2^2q/K^2)\log(d_x)}{T}} + \frac{L_0q\log(d_x)}{T}\\
        &\qquad+\frac{L_0\sqrt{(TK/q + qK)\log(1/\delta)}\log(d_x)}{n\varepsilon}  +\frac{L_2}{K} + \frac{L_1q\log(T/q)}{T}.
    \end{align*}
    Note that the sparsification error is, by Corollary \ref{cor:maurey_over_simplex}, $\mathbb{E}[F_\D(\hat{w}^T) - F_\D(w^T)] \lesssim \frac{L_1}{K}$. Combining this with the previous inequality and setting $ q = \sqrt{T/\log(d_x)}, K = T/q = \sqrt{T\log(d_x)}$, leads to 
    \begin{multline*}
        \textstyle\mathbb{E}[F_\D(\hat{w}^T) - F_\D(x)] \lesssim (L_0+L_1)\sqrt{\frac{\log(d_x)}{T}} \\
        \textstyle + \frac{L_2}{T^{3/4}\log(d_x)^{1/4}} +\frac{L_1 + L_2}{\sqrt{T\log(d_x)}} 
        + \frac{L_1\log(\sqrt{n\log(d_x)})}{\sqrt{T\log(d_x)}}+\frac{L_0\sqrt{T\log(1/\delta)}\log ^{3/2}(d_x)}{n\varepsilon}\\
        \textstyle\lesssim (L_0+L_1)\sqrt{\frac{\log(d_x)}{T}} +\frac{L_2}{\sqrt{T}\log(d_x)^{1/4}}  
        + \frac{L_1\log(\sqrt{n\log(d_x)})}{\sqrt{T\log(d_x)}}+\frac{L_0\sqrt{T\log(1/\delta)}\log ^{3/2}(d_x)}{n\varepsilon}.
    \end{multline*}
    Hence letting $T = \min\Big\{n,\frac{n\varepsilon}{\log(d_x)\sqrt{\log(1/\delta)}}\Big\}$ gives the claimed rate for second-order-smooth functions. Now, we look at the case where the partial derivates are not necessarily smooth. Now, we can bound $\mathbb{E}[F_\D(w^T) - F_\D(x)]$ by
    \begin{align*}
        &\frac{\log(d_x)}{\tau T} + \frac{\tau L_0^2}{2} + \frac{1}{T}\sum_{t \in [T]} \mathbb{E}\left[\langle \nabla F_\D(w^t) - g^t, x^t-x\rangle\right]  \\
        &\lesssim \frac{\log(d_x)}{\tau T} + \tau L_0^2 + \frac{1}{T}\sum_{t \in [T]} \bigg (\frac{L_1}{\sqrt{K}} + \mathbbm{1}_{(t > q)}\bigg[\frac{4L_1q}{q\lfloor t/q\rfloor + 1} \\
        & \qquad \qquad \qquad \qquad \qquad \qquad \qquad + \frac{L_1^2q\tau}{\sqrt{\log(d_x)}K^{3/2}} + \frac{(L_0^2 + L_1^2)\sqrt{\log(d_x)}q\tau}{\sqrt{K}}\bigg]\bigg) \\
        &\lesssim \frac{\log(d_x)}{\tau T} + \tau\Big[ L_0^2 + \frac{L_1^2q}{\sqrt{\log(d_x)}K^{3/2}} + \frac{(L_0^2 + L_1^2)\sqrt{\log(d_x)}q}{\sqrt{K}} \Big] \\
        & \qquad \qquad  + \frac{L_1}{\sqrt{K}} + \frac{1}{T}\sum_{t = q+1}^T \left[\frac{L_1q}{q\lfloor t/q\rfloor + 1}\right]\\
        &\lesssim \frac{\log(d_x)}{\tau T} + \tau\Big[ L_0^2 + \frac{L_1^2q}{\sqrt{\log(d_x)}K^{3/2}} + \frac{(L_0^2 + L_1^2)\sqrt{\log(d_x)}q}{\sqrt{K}} \Big]  + \frac{L_1}{\sqrt{K}} + \frac{L_1 q \log(T/q)}{T}.
    \end{align*}
    By Corollary \ref{cor:maurey_over_simplex}, the sparsification error is $\mathbb{E}[F_\D(\hat{w}^T) - F_\D(w^T)] \lesssim \frac{L_1}{K}$, which implies
    \begin{align*}
    \mathbb{E}[F_\D(\hat w^T) - F_\D(x)] \lesssim & \frac{\log(d_x)}{\tau T} + \tau\Big[ L_0^2 + \frac{L_1^2q}{\sqrt{\log(d_x)}K^{3/2}} + \frac{(L_0^2 + L_1^2)\sqrt{\log(d_x)}q}{\sqrt{K}} \Big] \\
    & \quad + \frac{L_1}{\sqrt{K}} + \frac{L_1 q \log(T/q)}{T}.
    \end{align*}    
    Setting $T = \min\left\{n,\frac{(n\varepsilon)^{4/5}}{(\log(d_x)\log(1/\delta))^{2/5}}\right\}$, $ q = \sqrt{T}/\log(d_x)$, $K = T/\log(d_x)$ and \\ $\tau \! \eqsim \! \!\min\!\left\{\sqrt{\frac{\log(d_x)}{(L_0^2 + (L_0^2+L_1^2)q\sqrt{\log(d_x)/K} + L_1^2q/[\sqrt{\log(d_x)}K^{3/2}])T}},  \frac{1}{L_0q}, \frac{n\varepsilon}{TL_0\sqrt{(TK/q + qK)\log(1/\delta)}}\right\}$ leads to the convergence rate claimed for first-order-smooth functions.  
\end{proof}

\section{Conclusion and Future Work} We presented mirror descent based algorithms for DP-SSP in the polyhedral setting. 
Our algorithms are the first achieving polylogarithmic-in-the-dimension error beyond bilinear objectives. In fact, for
second-order-smooth objectives our rates are optimal up to logarithmic factors. 
An important question remains: Are the optimal rates achievable in expectation for first-order-smooth
functions? We leave this question for future work. We also apply our ideas to DP-SCO in the polyhedral setting, where we also leave open the problem of achieving optimal rates with mirror descent based algorithms. 

\section{Acknowledgments} T.~Gonz\'alez's research was partially done as a Research Intern at Google DeepMind.  C.~Guzm\'an's research was partially done as a Visiting Researcher at Google. C.~Guzm\'an's research was partially supported by INRIA Associate Teams project, ANID FONDECYT 1210362 grant, ANID Anillo ACT210005 grant, and National Center for Artificial Intelligence CENIA FB210017, Basal ANID. C. Paquette is a CIFAR AI chair, (MILA) and C. Paquette was supported by a NSERC Discovery Grant, NSERC CREATE grant, an external grant from Google, and FRQNT New University Researcher's Start-Up Program.

\bibliographystyle{siamplain}
\bibliography{references}

\end{document}